\documentclass[10pt]{article}  
\usepackage{graphicx}
\usepackage{amscd}
\usepackage{amsmath,amsfonts,amssymb,amsthm,mathrsfs,cite}
\usepackage[all,pdf]{xy}
\xyoption{curve}
\usepackage{fancybox}
\usepackage{mathptm}
\usepackage{color}
\usepackage{url}
\newtheorem{thm}{Theorem}[section]

\newtheorem{cor}[thm]{Corollary}
\newtheorem{lem}[thm]{Lemma}
\newtheorem{prop}[thm]{Proposition}
\newtheorem{defn}[thm]{Definition}
\newtheorem{ques}[thm]{Question}
\theoremstyle{definition}

\newtheorem{exm}[thm]{Example}
\newtheorem{rem}[thm]{Remark}
\theoremstyle{default}
\newcounter{example}[section]

\newcommand{\m}{\Lambda}
\newcommand{\Hom}{\operatorname{Hom}}
\newcommand{\End}{\operatorname{End}}
\newcommand{\Ker}{\operatorname{Ker}}

\newcommand{\cok}{\operatorname{Coker}}

\newcommand{\D}{\operatorname{D}}

\newcommand{\cpx}[1]{#1^{\bullet}}
\newcommand{\opp}{^{\rm op}}
\newcommand{\add}{{\rm add}}
\newcommand{\Ext}{{\rm Ext}}
\newcommand{\Tor}{{\rm Tor}}
\DeclareMathOperator{\Mon}{{\sf Mon}}
\DeclareMathOperator{\stMon}{\underline{\sf Mon}}
\newcommand{\Rep}{{\rm Rep}}
\newcommand{\rad}{{\rm rad}}
\newcommand{\gldim}{{\rm gldim}}
\newcommand{\pd}{{\rm pd}}
\newcommand{\id}{{\rm id}}
\newcommand{\modcat}[1]{#1\mbox{\sf -mod}}

\newcommand{\Modcat}[1]{#1\mbox{\sf -Mod}}
\newcommand{\projcat}[1]{#1\mbox{\sf -proj}}
\newcommand{\Projcat}[1]{#1\mbox{\rm -Proj}}

\newcommand{\Gp}{\mbox{\sf -Gproj}}
\newcommand{\stGp}{\mbox -{\underline{\sf Gproj}}}
\newcommand{\GP}{\mbox{\sf -GProj}}
\DeclareMathOperator{\filt}{{\sf filt}}
\DeclareMathOperator{\addfilt}{\overline{\sf filt}}
\DeclareMathOperator{\Filt}{{\sf Filt}}
\DeclareMathOperator{\AddFilt}{\overline{\sf Filt}}

\newcommand{\lra}{\longrightarrow}
\newcommand{\lraf}[1]{\stackrel{#1}{\lra}}
\newcommand{\ra}{\rightarrow}

\title{Gorenstein projective bimodules via monomorphism categories and filtration categories}
\author{Wei Hu, Xiu-Hua Luo, Bao-Lin Xiong\thanks{The corresponding author. Email: xiongbaolin@gmail.com}, \ Guodong Zhou}
\date{}

\begin{document}
 \maketitle

\renewcommand{\thefootnote}{\alph{footnote}}
\setcounter{footnote}{-1}
\setcounter{footnote}{-1}

\begin{abstract}
We generalize the monomorphism category from quiver (with monomial relations) to arbitrary finite dimensional algebras by a homological definition. Given two finite dimension algebras $A$ and $B$, we use the special monomorphism category $\Mon(B, A\Gp)$ to describe some Gorenstein projective bimodules over the tensor product of $A$ and $B$. If one of the two algebras is Gorenstein, we give a sufficient and necessary condition for $\Mon(B, A\Gp)$ being the category of all Gorenstein projective bimodules. In addition, if both $A$ and $B$ are Gorenstein, we can describe the category of all Gorenstein projective bimodules via filtration categories. Similarly, in this case, we get the same result for infinitely generated Gorenstein projective bimodules.
\end{abstract}


\section{Introduction}

Auslander\cite{Auslander1967} initiated Gorenstein homological algebra by introducing modules of G-dimension zero over a Noetherian commutative local ring, which coincides the maximal Cohen-Macaulay modules over Gorenstein commutative local ring. Later, Auslander and Bridger \cite{Auslander1969}  generalized these modules to two-sided Noetherian rings, which are now called Gorenstein projective modules. The notion of Gorenstein projective modules over an arbitrary ring was introduced by Enochs and Jenda in \cite{Enochs1995}. These modules identify with maximal Cohen-Macaulay modules over Gorenstein Noetherian rings in the work of Buchweitz in \cite{Buchweitz1987a}.

Not only the place of Gorenstein projective modules in Gorenstein homological algebra is the same as that of projective modules in homological algebra, but the Gorenstein projective modules (namely, maximal Cohen-Macaulay modules) also play an important role in singularity theory. Buchweitz in \cite{Buchweitz1987a} proved that the stable category of the category of Gorenstein projective modules over a Gorenstein Noetherian ring is triangle equivalent to the stable derived category which is just the singularity category defined by Orlov \cite{Orlov2004}. Note that Happel obtained the same result for Gorenstein algebras independently \cite{Happel1991b}. In particular, CM-finiteness (A ring is CM-finite if it has only finitely many isoclasses of indecomposable Gorenstein projective modules) is closely related to simple singularities, see \cite{Knorrer1987,Buchweitz1987b,Christensen2008}. For CM-finite Artin algebras, there is an Auslander-type result for Gorenstein projective modules \cite{Chen08}.

Thus, one of the most important tasks is to describe Gorenstein projective modules, especially in non-commutative case. Up to now, there are some partial results, mainly concentrated on Artin algebras, such as $T_2$-extension of an Artin algebra \cite{Li2010}, upper triangular matrix Artin algebra \cite{Xiong2012,Zhang2013,EshraghiLi2016}, Artin algebras with radical square zero \cite{Chen2012f,Ringel2012a}, Nakayama algebras \cite{Ringel2013b}, monomial algebras \cite{Chen2015}, and some cases of tensor product of two algebras---we will show it explicitly below. In this paper, we study Gorenstein projective modules over the tensor product of two algebras. This has been done only in some special cases. Li and Zhang's result of $T_2$-extension can be viewed as describing the Gorenstein projective modules over $A\otimes_kkQ$, where $Q$ is the Dynkin quiver of type $A_2$. Ringel and Zhang \cite{Ringel2011-preprint} completely described the Gorenstein projective modules of $kQ\otimes_kk[\epsilon]$, where $Q$ is a finite acyclic quiver and $k[\epsilon]=k[x]/(x^2)$. In \cite{Luo2013}, Luo and Zhang studied the Gorenstein projective modules over $A\otimes_kkQ$, where $Q$ is a finite acyclic quiver. Later, Luo and Zhang \cite{Luozhang2015} generalize it to $Q$ with monomial relations. The approaches of \cite{Li2010,Luo2013,Luozhang2015} mainly use monomorphism categories. In this paper, besides via monomorphism categories for which we give a general homological definition, we give another approach to the Gorenstein projective modules over tensor product of algebras, via filtration categories.

\medskip

-{\bf via monomorphism category}. The category of monomorphisms in a module category can be viewed as the first example of monomorphism category. This goes back to G.Birkhoff's problem \cite{Birkhoff1934} of classifying all subgroups of abelian $p$-groups by matrices. This question were related to representations of partial order sets, such as \cite{Arnold2000,Simson2002, Simson2018}. Ringel and Schmidmeier \cite{Ringel2006b,Ringel2008c,Ringel2008b} renewed this subject by studying the representation type and the Auslander-Reiten theory of the submodule category for an Artin algebra, and in particular for $k[x]/(x^n)$. Kussin, Lenzing and Meltzer \cite{Kussin2013a,Kussin2013b} related submodule categories to weighted projective lines and singularity theory. Chen \cite{Chen2011c,Chen2012d} studied these categories from the viewpoint of triangulated categories and Gorenstein homological algebra.

For an algebra $A$, the category of monomorphisms in the module category of $A$ can be viewed as a full subcategory of the module category of the tensor product $A\otimes_kkA_2$, where $A_2$ is the quiver $1\to 2$. Zhang \cite{Zhang2011} introduced monomorphism categories of type $A_n$, say the category of $n-1$ successive monomorphisms, which can be viewed as a full subcategory of the module category of $A\otimes_kkA_n$, where $A_n$ is the quiver $1\to 2\to\cdots\to n$. The Auslander-Reiten theory of this category was studied in \cite{Xiong2014}. In \cite{Luo2013}, Luo and Zhang generalized the above notion and introduced monomorphism categories over finite acyclic quivers, and then over finite acyclic quivers with monomial relations \cite{Luozhang2015}. Recently, the Ringel-Schmidmeier-Simson (RSS for short) equivalence on monomorphism categories is introduced by Zhang et al. \cite{Xiong2017, Xiong2018}. Based on the combinatorial information of the quiver (and the monomial relations), the monomorphism category was applied to describe Gorenstein projective modules over the tensor product of an algebra and a path algebra (modulo the ideal generated by the monomial relations) over a field.

\smallskip

In this paper, we define monomorphism categories over an arbitrary finite dimensional algebra, using homological conditions (see Definition \ref{Def: homological monic representations}). Our definition is equivalent to the previous ones when restricting to finite acyclic quivers (with monomial relations). The homological definition seems to simplify many arguments. For two finite dimensional $k$-algebras $A$ and $B$, we denote the monomorphism category of $B$ over $A$ by $\Mon(B, A)$. For a full additive subcategory of $A$-mod $\mathscr{L}$, we can also define the full subcategory $\Mon(B, \mathscr{L})$ of $\Mon(B, A)$. We show that $\Mon(B, A\Gp)$ is a Frobenius exact category and is a full subcategory of $\Lambda\Gp$ where $\Lambda=A\otimes_kB$ (see Proposition \ref{proposition-Frobenius} and \ref{proposition-monic-in-gorenstein}). Furthermore, we have the following main result.

\medskip

{\bf\noindent Theorem A} (Theorem \ref{theorem-gldimB-finite-monic=gorenstein} and \ref{theorem-Gorenstein-CM-free}) {\it Let $A$ and $B$ be two finite dimensional $k$-algebras. Assume either $A$ or $B$ is Gorenstein. Then $(A\otimes_kB)\Gp=\Mon(B, A\Gp)$ if and only if $B$ is CM-free.}

\smallskip

Thus once the modules in $A\Gp$ are known, it is relatively easy to describe the modules in $\Mon(B, A\Gp)$. This gives a satisfactory answer in this case.

\medskip

-{\bf via filtration category}. In the definition of monomorphism categories $\Mon(B, A)$, the role of $B$ and $A$ are not symmetric (see Example \ref{example-not-symmetric}). This motivates us to consider the filtration category $\addfilt(A\Gp\otimes_k B\Gp)$, that is, the class of $A\otimes_kB$-modules which are direct summands of iterated extensions of tensor products of Gorenstein projective $A$-modules and Gorenstein projective $B$-modules. The modules in the filtration category $\addfilt(A\Gp\otimes_k B\Gp)$ are always Gorenstein projective. The converse does not hold in general (see Example \ref{example-filtration-fail} below). However, in case that both algebras are Gorenstein, or one of the algebras is a triangular algebra, we obtain the following main results:

\medskip

{\bf\noindent Theorem B} (Theorem \ref{theorem-Gorenstein-tensor-Gorenstein}) {\it Let $A$ and $B$ be Gorenstein algebras. Assume that $k$ is a splitting field for $A$ or $B$. Then $$(A\otimes_kB)\Gp=\addfilt(A\Gp\otimes_k B\Gp).$$}

\medskip

{\bf\noindent Theorem C} (Theorem \ref{theorem-B-triangular}) {\it Let $A$ and $B$ be two finite dimensional $k$-algebras. Assume that $B$ is triangular and $k$ is a splitting field for $B$. Then the  three categories $(A\otimes_kB)\Gp$, $\Mon(B, A\Gp)$ and $\filt(A\Gp\otimes B\Gp)$  are the same, where $\filt(A\Gp\otimes B\Gp)$ is the smallest full subcategory of $\modcat{(A\otimes_kB)}$ closed under extensions.}

\vskip5pt
We can even show that a similar result holds for infinitely generated Gorenstein projective modules. Its proof essentially uses Quillen's powerful small object argument.

\medskip

This paper is organized as follows. Section 2 contains some preliminary materials, including basic facts about Gorenstein projective modules. In Section 3, we introduce our (homological version of) monomorphism categories and study their basic properties. The next two sections are the heart of this paper. We describe Gorenstein projective bimodules using monomorphism categories and using filtration categories. Finally, we deal with infinitely generated Gorenstein projective bimodules by Quillen's small object argument in Subsection 5.2.

We should remark that Dawei Shen \cite{Shen2016} also obtained some results  of this paper independently, in particular, Proposition~\ref{Prop-gorenstein-projective-bimodules} (2) and Theorem~\ref{theorem-gldimB-finite-monic=gorenstein}.

\section {Preliminaries}

In this section, we recall some basic definitions and facts that needed in our later proofs.

\subsection{Convention}
Throughout this paper, $k$ is a fixed field, and all algebras are finite dimensional $k$-algebras. For simplicity, the tensor product $\otimes_k$ will be written as $\otimes$.

Let $A$ be an algebra. By $\modcat{A}$ we denote the category of finitely generated left $A$-modules, and by {\sf mod}-$A$ the category of finitely generated right $A$-modules. Unless specified otherwise, all modules are finitely generated left modules. The full subcategory of $\modcat{A}$ consisting of projective modules is denoted by $\projcat{A}$.  The functor $\D:=\Hom_k(-, k)$ is the usual duality, and the functor $\Hom_A(-, {}_AA): \modcat{A}\lra\modcat{A\opp}$ will be denoted by $(-)^*$.  The projective dimension of an $A$-module $X$ is denoted by $\pd ({}_AX)$. The global dimension of $A$, which is the supremum of the projective dimensions of all $A$-modules, is denoted by $\gldim(A)$.

A complex over $\modcat{A}$ is a sequence of $A$-module homomorphisms
$\cpx{X}: \cdots\lra X^i\lraf{d^i}X^{i+1}\lraf{d^{i+1}}X^{i+2}\lra\cdots$ such that $d^{i+1}d^i=0$ for all $i\in\mathbb{Z}$.  The $i$-th homology of $\cpx{X}$ is denoted by $H^i(\cpx{X})$. Let $M_A$ be a right $A$-module, we use $M\otimes_A\cpx{X}$ to denote the complex $\cdots\lra M\otimes_AX^i\lraf{1\otimes d^i}M\otimes_AX^{i+1}\lra\cdots$.  Similarly, for a left $A$-module ${}_AN$, we denote by $\Hom_A(\cpx{X}, {}_AN)$ the complex $\cdots\lra\Hom_A(X^{i+1}, {}_AN)\lraf{\Hom_A(d^{i}, N)}\Hom_A(X^i, N)\lra\cdots$. Note that the $i$-th term of $\Hom_A(\cpx{X}, N)$ is $\Hom_A(X^{-i}, N)$ for all $i\in\mathbb{Z}$.

\medskip
The following lemma collects some basic homological facts needed frequently in our proofs.

\begin{lem}\label{lemma-basic-homological-isomorphism}
 Let $A$ and $B$ be finite dimensional algebras.  Suppose that ${}_AX\in\modcat{A}$, $Y_A\in\mbox{\sf mod-}A$, ${}_BU\in\modcat{B}$, $V_B\in\mbox{\sf mod-}B$ and ${}_BM_A$ is a $B$-$A$-bimodule.  Then we have the following.
\begin{itemize}\setlength\itemsep{0pt}
\item[$(1)$.] There is a natural isomorphism $\D\Tor_i^A(X, Y)\cong \Ext_A^i(Y, \D(X))$ for all $i\geq 0$.

\item[$(2)$.] If $\Tor_i^A(M, X)=0=\Tor_i^B(V, M)$ for all $i>0$, then $\Tor_i^B(V, M\otimes_AX)\cong\Tor_i^A(V\otimes_BM, X)$ for all $i\geq 0$.

\item[$(3)$.] If $\Tor_i^A(M, X)=0=\Ext_B^i(M, U)$ for all $i>0$, then $\Ext_B^i(M\otimes_AX, U)\cong\Ext_A^i(X, \Hom_B(M, U))$ for all $i\geq 0$.
\end{itemize}
\end{lem}

\begin{proof}
(1) is well-known. (2) and (3) can be found, for example, in \cite[Chapter IX, Theorem 2.8,  2.8a]{Cartan1999}
\end{proof}

\subsection{Tensor products of algebras}
In this subsection, we collect some basic facts on modules over tensor products of algebras.

\begin{lem}\label{lemma-tensorModuleHom}
Let $A$ and $B$ be finite dimensional $k$-algebras, and let $\Lambda:=A\otimes B$. Then we have the following.
\begin{itemize}\setlength\itemsep{0pt}
 \item[$(1)$.]  ${\rm Hom}_{\m}(X_1\otimes Y_1, X_2\otimes Y_2)\cong {\rm Hom}_{A}(X_1, X_2)\otimes {\rm Hom}_{B}(Y_1, Y_2)$, where $X_1, \ X_2\in \modcat{A}$, $Y_1, \ Y_2\in \modcat{B}$.

 \item[$(2)$.] $\D(X\otimes Y)\cong \D(X)\otimes  \D(Y)$, where $X\in \modcat{A}$, $Y\in\modcat{B}$.

 \item[$(3)$.] $(W\otimes V)\otimes_{A\otimes  B}(X\otimes  Y)\cong (W\otimes_A X)\otimes (V\otimes_BY)$, where $W\in\mbox{\sf mod-}A$, $V\in\mbox{\sf mod-}B$, $X\in\modcat{A}$ and $Y\in\modcat{B}$.

 \item[$(4)$.] Let $V$ be a right $B$-module and $X$ a left $\Lambda$-module. Then there exists a natural isomorphism  of left $A$-modules
$$(A\otimes  V)\otimes_\Lambda X\cong   {}_A(V\otimes_B X), \  \  (a\otimes v)\otimes x\mapsto v\otimes (a\otimes 1)x.$$

 \item[$(5)$.] Let $V$ be a left $B$-module, $T$ a left $A$-module and $X$ a left $\Lambda$-module. Then there exists a natural isomorphism
$$\Hom_A((A\otimes  \D(V))\otimes_\Lambda X, T)\cong \Hom_\Lambda (X, T\otimes  V),$$
that is , there exists an adjoint pair $((A\otimes  \D(V))\otimes_\Lambda -, -\otimes  V)$.

 \item[$(6)$.]  Let $P$ be a finitely generated projective left $B$-module. Then there exists natural isomorphisms of functors
 $$(A\otimes  P^*)\otimes_\Lambda -\cong {}_A(P^*\otimes_B -)\cong {}_A\Hom_B(P, -): \modcat{\Lambda}\ra \modcat{A}$$
 and these functors are right adjoint to the functor  $-\otimes P: \modcat{A}\ra\modcat{\Lambda}$, that is,  for any $T\in\modcat{A}$ and $X\in\modcat{\m}$, there are functorial isomorphisms
$$\Hom_{\m}(T\otimes P, \, X)\cong\Hom_A(T, (A\otimes P^*)\otimes_{\m}X)\cong \Hom_A(T,  P^* \otimes_{B}X)\cong \Hom_A(T, \Hom_B(P, X)).$$
\end{itemize}
 \end{lem}

\begin{proof}
(1)-(3) are standard, their proofs can be found, for example, in \cite[Chapter XI]{Cartan1999}.

$(4)$. One can verify that the map is a well defined $A$-module homomorphism  and its inverse map is given by $v\otimes x\mapsto (1\otimes v)\otimes x$.

$(5)$.\  This follows from the series of isomorphisms
$$\begin{array}{rcl}\Hom_A((A\otimes  \D(V))\otimes_\Lambda X, T)&\cong& \Hom_\Lambda (X, \Hom_A(A\otimes  \D(V), T))\\
&\cong&
\Hom_\Lambda (X, \Hom_k(\D(V), \Hom_A(A, T)))\\
&\cong & \Hom_\Lambda (X, \Hom_k(\D(V), T))\\
&\cong& \Hom_\Lambda (X,  T\otimes  V).\end{array} $$

$(6)$.\  The isomorphism $(A\otimes  P^*)\otimes_\Lambda -\cong {}_A(P^*\otimes_B -): \modcat{\Lambda}\ra \modcat{A}$ follows from (4) and the isomorphism $ {}_A(P^*\otimes_B -)\cong {}_A\Hom_B(P, -)$ holds because $P$ is a finitely generated left $B$-module.

It follows from (1) that $A\otimes P^*\cong\Hom_{\m}(A\otimes P, \m)$. Then $(A\otimes P^*)\otimes_{\m}X\cong\Hom_{\m}(A\otimes P, \m)\otimes_{\m}X\cong\Hom_{\m}(A\otimes P, X)$. Hence we have the following isomorphisms
$$\begin{array}{ll}
\Hom_A(T, (A\otimes P^*)\otimes_{\m}X) & \cong\Hom_A(T, \Hom_{\m}(A\otimes P, X))\\
& \cong\Hom_{\m}((A\otimes P)\otimes_A  T,\, X) \\
& \cong\Hom_{\m}((A\otimes_A  T)\otimes P,\, X) \\
& \cong\Hom_{\m}(T\otimes P,\, X).
\end{array}$$
This finishes the proof.
\end{proof}

The following result is well-known. However, it seems hard to find a precise reference in the literature. Here we provide a proof
 for the convenience of the reader.

\begin{lem}\label{Lem-modules-over-tensor-product}
Let $A$ and $B$ be finite dimensional $k$-algebras, and let $\Lambda:=A\otimes B$.
\begin{itemize}\setlength\itemsep{0pt}
\item[$(1)$.] If ${}_AX\in\modcat{A}$ and ${}_BY\in\modcat{B}$ are projective (respectively, injective), then so is the $\m$-module $X\otimes Y$.

\item[$(2)$.]  Suppose that $k$ is a splitting field for $B$, that is, every simple $B$-module has endomorphism algebra $k$.  If both ${}_AX\in\modcat{A}$ and ${}_BY\in\modcat{B}$ are indecomposable projective (respectively, indecomposable injective, simple), then $X\otimes Y$ is an indecomposable projective (respectively, indecomposable injective, simple) $\m$-module.

 \item[$(3)$.]  Suppose that $k$ is a splitting field for $B$. Then all the indecomposable projective (respectively, indecomposable injective, simple) $\m$-modules are of the form $X\otimes Y$, where ${}_AX\in\modcat{A}$ and ${}_BY\in\modcat{B}$ are indecomposable projective (respectively, indecomposable injective, simple) modules.
\end{itemize}
\end{lem}

\begin{proof}
$(1)$. The statement  is clear.

$(2)$.  Suppose that $_AX$ is an indecomposable projective $A$-module, and $_BY$ is an indecomposable projective $B$-module. Since $k$ is a splitting field for $B$, the endomorphism algebra $\End_B(Y)$ is a finite dimensional local algebra with
$$\End_B(Y)/\rad\big(\End_B(Y)\big)\cong \End_B\big(Y/\rad(Y)\big)\cong k.$$
By Lemma \ref{lemma-tensorModuleHom} (1), there is an algebra isomorphism $\End_{\m}(X\otimes Y)\cong\End_A(X)\otimes \End_B(Y)$. For simplicity, we write $E_X$ and $E_Y$ for $\End_A(X)$ and $\End_B(Y)$, respectively. Note that the ideal
$$I:=E_X\otimes\rad(E_Y)+\rad(E_X)\otimes E_Y$$
 of $E_X\otimes E_Y$ is nilpotent, and by the $3\times 3$ diagram obtained by tensoring the short exact sequences $0\ra \rad(E_X)\ra E_X\ra E_X/\rad(E_X)\ra 0$ and $0\ra \rad(E_Y)\ra E_Y\ra E_Y/\rad(E_Y)\ra 0$ over $k$, one see that
 $$(E_X\otimes E_Y)/I\cong E_X/\rad(E_X)\otimes E_Y/\rad(E_Y)\cong  E_X/\rad(E_X)$$
is a local algebra.  It follows that $\End_{\m}(X\otimes Y)$ is a local algebra. This implies that $X\otimes Y$ is an indecomposable $\m$-module.

\smallskip
 When $X$ and $Y$ are injective. The proof can be completed by combining (1) and Lemma \ref{lemma-tensorModuleHom}(2).

\smallskip
   Now suppose that $X$ and $Y$ are simple. Then $$\End_{\m}(X\otimes Y)\cong \End_A(X)\otimes \End_B(Y)\cong \End_A(X)$$ is a division algebra. Hence $X\otimes Y$ is an indecomposable $\m$-module.  Note that $B/\rad(B)\cong M_{n_1}(k)\times\cdots M_{n_r}(k)$ is a direct product of full matrix algebras over $k$, since $k$ is a splitting field for $B$. Thus, the algebra $\bar{\m}:=(A/\rad(A))\otimes (B/\rad(B))$, which is the quotient algebra of $\m$ modulo the ideal $\rad(A)\otimes B+A\otimes \rad(B)$, is Morita equivalent to a direct product of copies of $A/\rad(A)$ which is semi-simple. This implies that $$\rad(\m)=\rad(A)\otimes B+A\otimes \rad(B).$$  It is easy to see that $X\otimes Y$ is annihilated by the ideal $\rad(A)\otimes B+A\otimes \rad(B)$, indicating that $X\otimes Y$ is a semi-simple $\m$-module. We have already shown that $X\otimes Y$ is indecomposable. Therefore $X\otimes Y$ is a simple $\m$-module.

$(3)$.  Write
$${}_AA=\bigoplus_{i=1}^nX_i, \quad {}_BB=\bigoplus_{j=1}^mY_j$$
as direct sums of indecomposable projective modules. Then
$$_{\m}\m=A\otimes B=\bigoplus_{i=1}^n\bigoplus_{j=1}^mX_i\otimes Y_j.$$
For each pair of  $i, j$, the $\Lambda$-module $X_i\otimes Y_j$ is  indecomposable and projective by (2).
 Since every indecomposable projective $\m$-module is isomorphic to a direct summand of $_{\m}\m$, we are done. The proof for injective modules is similar.  Each simple $\m$-module must be the unique simple quotient of an indecomposable projective $\m$-module, say $X\otimes Y$, where ${}_AX$ and ${}_BY$ are indecomposable projective. The unique simple quotient of $X\otimes Y$ is $X/\rad(X)\otimes Y/\rad(Y)$ by (2). This finishes the proof.
\end{proof}

\subsection{Gorenstein projective modules}
Let $A$ be an algebra, and let $G$ be an $A$-module.
A {\em complete projective resolution} of $G$ is an exact complex
$$\cpx{P}: \quad\cdots\lra P^{i}\lraf{d^i} P^{i+1}\lraf{d^{i+1}} P^{i+2}\lra\cdots$$
 of modules in $\projcat{A}$ with $G=\Ker\, d^1$ such that $\Hom_A(\cpx{P}, {}_AA)$ is again exact.  If $G$ admits a complete projective resolution, then $G$ is called a {\em Gorenstein projective} $A$-module. The full subcategory of $A\modcat$ consisting of all Gorenstein projective $A$-modules is denoted by $A\Gp$.

\smallskip
The following lemma is well-known; for the convenience of the reader, we give a proof.

\begin{lem} \label{lemma-complete-resolution-basicProperty}
 Let $A$ and $B$ be algebras, and let $\cpx{P}$ be an exact complex of projective left $A$-modules.
\begin{itemize}
  \setlength\itemsep{0pt}
  \item[$(1)$.] If $M_A\in\mbox{\rm mod-}A$ has finite projective dimension,  then $M\otimes_A\cpx{P}$ is exact.

 \item[$(2)$.] If ${}_AN\in\modcat{A}$ has finite injective dimension, then $\Hom_A(\cpx{P}, N)$ is exact.

  \item[$(3)$.] If $\Hom_A(\cpx{P}, {}_AA)$ is exact and $_AN\in\modcat{A}$ has finite projective dimension, then $\Hom_A(\cpx{P}, N)$ is exact.

  \item[$(4)$.] If $\Hom_A(\cpx{P}, {}_AA)$ is exact and $M_A\in\mbox{\rm mod-}A$ has finite injective dimension, then $M\otimes_A\cpx{P}$ is exact.
\end{itemize}
\end{lem}

\begin{proof}
(1). We use induction of the projective dimension of $M_A$. If $\pd(M_A)=0$, then $M_A$ is projective. It follows that  $M\otimes_A\cpx{P}$ is exact. Now assume that $\pd(M_A)>0$. Let $0\lra \Omega(M)\lra Q_M\lra M\lra 0$
be an exact sequence with $Q_M$ projective. Then $\pd(\Omega(M))<\pd(M_A)$. By induction hypothesis, we can assume that $\Omega(M)\otimes_A\cpx{P}$ is exact.
The projectivity of $Q_M$ implies the exactness of  $Q_M\otimes_A\cpx{P}$. It follows that $M\otimes_A\cpx{P}$ has zero homology in all degrees by the long exact sequence of homology induced by the short exact sequence $0\ra \Omega(M)\otimes_A\cpx{P}\ra Q_M\otimes_A\cpx{P}\ra M\otimes_A\cpx{P}\ra 0$ of complexes.  This proves (1).

The proofs of (2) and (3) are similar.

 For the statement (4), we consider the complex $\D(M\otimes_A\cpx{P})$, which is isomorphic to $\Hom_A(\cpx{P}, \D(M_A))$. Since $M_A$ has finite injective dimension, the left $A$-module $\D(M_A)$ has finite projective dimension. It follows from (3) that $\D(M\otimes_A\cpx{P})$ is exact. Hence $M\otimes_A\cpx{P}$ is exact.
 \end{proof}

From the above lemma, we can easily prove the following assertions.

\begin{lem}\label{lemma-Gorenstein-Tor-Ext-property}
Let $A$ and $B$ be algebras, and let ${}_AG$ be a Gorenstein projective $A$-module. Suppose that ${}_AN$ is a left $A$-module, $L_A$ is a right $A$-module and ${}_BM_A$ is a $B$-$A$-bimodule.
\begin{itemize}
 \setlength\itemsep{0pt}
 \item[$(1)$.] If $\pd ({}_AN)<\infty$ or $\id({}_AN)<\infty$, then   $\Ext_A^i(G, N)=0$ for all $i\ge1$.

 \item[$(2)$.] If $\pd(L_A)<\infty$ or $\id(L_A)<\infty$, then $\Tor_i^A(L, G)=0$ for all $i\ge1$.

 \item[$(3)$.] Suppose that ${}_BM$ is projective, $\pd(M_A)<\infty$ and $\pd({}_A\Hom_B(M, B))<\infty$. Then ${}_BM\otimes_AG$ is a Gorenstein projective $B$-module.
\end{itemize}
\end{lem}

\begin{proof}
The statements $(1)$ and $(2)$ follow from Lemma \ref{lemma-complete-resolution-basicProperty}.

 $(3)$.  Let $\cpx{P}$ be a complete projective resolution of $G$. Then by Lemma \ref{lemma-complete-resolution-basicProperty} (1), the complex ${}_BM\otimes_A\cpx{P}$ is again exact since $\pd(M_A)<\infty$.  Since ${}_BM$ is projective, all terms of $M\otimes_A\cpx{P}$ are projective as left $B$-modules. Finally, the adjointness gives rise to a natural isomorphism
 $$\Hom_B({}_BM\otimes_A\cpx{P}, B)\cong \Hom_A(\cpx{P}, \Hom_B(M, B)).$$
 Since ${}_A\Hom_B(M, B)$ has finite projective dimension, the complex $\Hom_A(\cpx{P}, \Hom_B(M, B))$ is exact by Lemma \ref{lemma-complete-resolution-basicProperty} (3). Hence $\Hom_B({}_BM\otimes_A\cpx{P}, B)$ is exact, and therefore ${}_BM\otimes_A\cpx{P}$ is a complete projective resolution of ${}_BM\otimes_AG$. This implies that ${}_BM\otimes_AG$ is a Gorenstein projective $B$-module.
\end{proof}

 Lemma \ref{lemma-complete-resolution-basicProperty} also gives rise to the following result  concerning Gorenstein projective modules over tensor products of algebras. Recall that a finite dimensional algebra $\Gamma$ is Gorenstein provided that $\id({}_{\Gamma}\Gamma)<\infty$ and $\id(\Gamma_{\Gamma})<\infty$. In this case, it is well-known that $\id({}_{\Gamma}\Gamma)=\id(\Gamma_{\Gamma})$.

\begin{prop}\label{Prop-gorenstein-projective-bimodules}
Let $A$ and $B$ be algebra, and let $\Lambda:=A\otimes B$. Then we have the following.
\begin{itemize}
   \setlength\itemsep{0mm}
 \item[$(1)$.] Suppose that $X\in A\Gp$ and $Y\in B\Gp$. Then $X\otimes Y\in\Lambda\Gp$.

 \item[$(2)$.] Suppose that $G$ is a Gorenstein projective $\Lambda$-module. If $A$ or $B$ is a Gorenstein algebra, then both ${}_AG$ and ${}_BG$ are Gorenstein projective.
\end{itemize}
\end{prop}

\begin{proof}
$(1)$. Let $P^\bullet $ and $Q^\bullet$ be complete resolutions of $X$ and $Y$ respectively. Denote
$$\tau_{\leq 0}P^\bullet=\big(\cdots \to P^{-1}\to P^0\big) \quad \mbox{and}
\quad \tau_{\geq 1}P^\bullet=\big( P^1\to P^2\to \cdots\big).$$
Similarly, we define $\tau_{\leq 0}Q^\bullet$ and $\tau_{\geq 1}Q^\bullet$.
Thus, the complex $(\tau_{\leq 0}\cpx{P})\otimes  (\tau_{\leq 0}\cpx{Q})$ is a projective resolution of $X\otimes  Y$ and
the complex $(\tau_{\geq 1}P^\bullet)\otimes  (\tau_{\geq 1}Q^\bullet)$  is a projective coresolution of $X\otimes  Y$. Combining these two complexes, we get an exact complex, denoted by $R^\bullet$,  whose first  cocycle module is $X\otimes  Y$.

We need to show that $\Hom_{\Lambda}(R^\bullet, A\otimes  B)$ is exact.  Note that $\Hom_{\Lambda}(R^\bullet, A\otimes  B)$ can be obtained by combining
$$\Hom_{\Lambda}((\tau_{\leq 0}P^\bullet)\otimes  (\tau_{\leq 0}Q^\bullet), A\otimes  B)\cong \Hom_{A}(\tau_{\leq 0}P^\bullet, A)\otimes  \Hom_{B}(\tau_{\leq 0}Q^\bullet, B) $$ and
$$\Hom_{\Lambda}((\tau_{\geq 1}P^\bullet)\otimes  (\tau_{\geq 1}Q^\bullet), A\otimes  B)\cong \Hom_{A}(\tau_{\geq 1}P^\bullet, A)\otimes  \Hom_{B}(\tau_{\geq  1}Q^\bullet, B). $$
The complex $\Hom_{A}(\tau_{\leq 0}P^\bullet, A)\otimes  \Hom_{B}(\tau_{\leq 0}Q^\bullet, B) $ is a coresolution of $\Hom_A(X, A)\otimes  \Hom_B(Y,  B)$, and the complex $\Hom_{A}(\tau_{\geq 1}P^\bullet, A)\otimes  \Hom_{B}(\tau_{\geq  1}Q^\bullet, B)$ is a resolution of $\Hom_A(X, A)\otimes  \Hom_B(Y,  B)$. Hence, $\Hom_{\Lambda}(R^\bullet, A\otimes  B)$ is exact.

\smallskip

$(2)$.  Without loss of generality, we assume that $A$ is a Gorenstein algebra. Then $\id ({}_AA)=\id(A_A)<\infty$. Equivalently $\pd\, \D(A_A)=\pd\, \D({}_AA)<\infty$.  Let $\cpx{P}$ be a complete projective resolution of $G$ over $\Lambda$. Then ${}_A\cpx{P}$ is an exact sequence of projective $A$-modules and ${}_AG$ is the cokernel of the differential $P^{-1}\ra P^0$. To show that ${}_AG$ is Gorenstein, it suffices to prove that ${}_A\cpx{P}$ is $\Hom_A(-, A)$-exact. Note that
$$\Hom_A({}_A\cpx{P}, A)\cong\Hom_A(\Lambda\otimes_{\Lambda}\cpx{P}, A)\cong\Hom_{\Lambda}(\cpx{P}, \Hom_A(\Lambda, A))\cong\Hom_{\Lambda}(\cpx{P}, A\otimes \D(B_B)).$$
Since $\id({}_AA)<\infty$, the left $\Lambda$-modules $A\otimes \D(B_B)$ has finite injective dimension. By Lemma \ref{lemma-complete-resolution-basicProperty} (2), the complex $\Hom_{\Lambda}(\cpx{P}, A\otimes \D(B_B))$ is exact. Hence $\Hom_A({}_A\cpx{P}, A)$ is exact, and thus ${}_A\cpx{P}$ is a complete projective resolution of ${}_AG$. This implies that ${}_AG$ is Gorenstein projective.  Similarly, ${}_B\cpx{P}$ is an exact sequence of projective $B$-modules, and there is an isomorphism of complexes $\Hom_{B}({}_B\cpx{P}, B)\cong\Hom_{\Lambda}(\cpx{P}, \D(A_A)\otimes B)$ by adjointness. Since $\pd\D(A_A)$ has finite projective dimension, so does the left $\Lambda$-module $\D(A_A)\otimes B$. By Lemma \ref{lemma-complete-resolution-basicProperty} (3), we deduce that $\Hom_{\Lambda}(\cpx{P}, \D(A_A)\otimes B)$ is exact. It follows that  ${}_B\cpx{P}$ is $\Hom_B(-, B)$-exact and is a complete projective resolution of ${}_BG$. Hence ${}_BG$ is Gorenstein projective.
\end{proof}

\begin{rem}
In a recent preprint {\rm \cite{Shen2016}},  Dawei Shen gives a characterization of Gorenstein projective bimodules when $A$ or $B$ is Gorenstein.  Proposition~\ref{Prop-gorenstein-projective-bimodules}   (2) is  a consequence of his result.
\end{rem}


\section {Monomorphism categories}

In this section, starting from Luo and Zhang's definition \cite{Luo2013} of monic representations of an acyclic quiver over algebra, we introduce monic representations of an arbitrary finite dimensional algebra over algebra (Definition \ref{Def: homological monic representations}), and study the category of all such monic representations and its full subcategories.

\subsection{Definitions}

We first recall the definition of monic representations of a quiver $Q$ over an algebra $A$ introduced in \cite{Luo2013}.  Let $Q=(Q_0, Q_1)$ be a quiver,  that is, a directed graph for which  $Q_0$ is the set  of vertices and $Q_1$ is the set of arrows between vertices. We shall always assume that $Q$ is a finite quiver, i.e. $Q_0$ and $Q_1$ are finite sets. The starting vertex of a path $p$ is $s(p)$, and the ending vertex of $p$ is $t(p)$. The trivial path corresponding to a vertex $i\in Q_0$ is denoted by $e_i$. Let $A$ be a finite dimensional $k$-algebra.   A {\em representation $X$ of $Q$ over $A$} is a datum
$$X=(X_i, X_{\alpha}, i\in Q_0, \alpha\in Q_1),$$
where $X_i\in \modcat{A}$ for all $i\in Q_0$, and $X_{\alpha}$ is an $A$-module homomorphism from $X_{s(\alpha)}$ to $X_{t(\alpha)}$ for all $\alpha\in Q_1$.  A morphism $f: X\longrightarrow Y$ between two representations $X$ and $Y$ of $Q$ over $A$ is a collection of $A$-module homomorphisms $\{f_i: X_i\rightarrow Y_i, i\in Q_0\}$ such that  $f_{t(\alpha)}X_{\alpha}=Y_{\alpha}f_{s(\alpha)}$ for all $\alpha\in Q_1$. A representation $X$ is called a {\em monic representation} of $Q$ over $A$ if, for each $i\in Q_0$, the map
$$(X_{\alpha})_{\begin{smallmatrix}\alpha\in Q_1\\t(\alpha)=i\end{smallmatrix}}: \bigoplus_{\begin{smallmatrix}{\alpha\in Q_1}\\{t(\alpha)=i}\end{smallmatrix}}X_{s(\alpha)}\longrightarrow X_i$$
is a monomorphism.  We denote by $\Rep(Q, A)$ the category of all representations of $Q$ over $A$, and denote by $\Mon(Q, A)$ its full subcategory consisting of all monic representations.

\smallskip
Let $\Lambda:=A\otimes kQ$. Then   $\Rep(Q, A)$ is actually equivalent to $\modcat{\Lambda}$, see \cite{Les94}; for more details, see \cite{Luo2013}.  In fact, for each representation $X$ of $Q$ over $A$, one can view the $A$-module $\bigoplus\limits_{i\in Q_0}X_i$ as a $\Lambda$-module.  The action of $\Lambda$  is as follows:  $$(a\otimes \alpha)\cdot (x_i)_{i\in Q_0}=a\cdot X_{\alpha}(x_{s(\alpha)})\quad\mbox{and}\quad (a\otimes e_j)\cdot (x_i)_{i\in Q_0}=a\cdot x_j$$
for all $a\in A$, $\alpha\in Q_1$, $j\in Q_0$ and $(x_i)_{i\in Q_0}\in \bigoplus\limits_{i\in Q_0}X_i$.

In the following  we shall identify a representation $X$  of $Q$ over $A$ with its corresponding  $\m$-module $\bigoplus\limits_{i\in Q_0}X_i$. Under this identification, for $i\in Q_0$, we have
$(1\otimes  e_i)X=X_i$.

For each $i\in Q_0$,  let $S_i$ be the corresponding simple left $kQ$-module.  Then the projective resolution of simple right $kQ$-module  $D(S_i)$ is as follows
$$0\lra\bigoplus_{\begin{smallmatrix}\alpha\in Q_1\\t(\alpha)=i\end{smallmatrix}}e_{s(\alpha)}kQ\lraf{(\alpha\cdot)} e_ikQ\lra D(S_i)\lra 0$$
The morphism $(\alpha\cdot)$ is given by multiplying $\alpha$ on the left side. Applying the exact functor $A\otimes -$, we get an exact sequence of  $A$-$\m$-bimodules
$$0\lra\bigoplus_{\begin{smallmatrix}\alpha\in Q_1\\t(\alpha)=i\end{smallmatrix}}A\otimes e_{s(\alpha)}kQ\lraf{((1\otimes \alpha)\cdot)} A\otimes e_ikQ\lra A\otimes D(S_i)\lra 0,$$
which is a projective resolution of the right $\m$-module $A\otimes D(S_i)$ and can be written as
$$0\lra\bigoplus_{\begin{smallmatrix}\alpha\in Q_1\\t(\alpha)=i\end{smallmatrix}}(1\otimes e_{s(\alpha)})\m\lraf{((1\otimes \alpha)\cdot)} (1\otimes e_i)\m\lra A\otimes D(S_i)\lra 0.$$
Given a representation $X$ of $Q$ over $A$, we apply $-\otimes_{\m}X$ to the above sequence, and get a sequence of left $A$-modules, $$0\lra\bigoplus_{\begin{smallmatrix}\alpha\in Q_1\\t(\alpha)=i\end{smallmatrix}}X_{s(\alpha)}\lraf{(X_{\alpha})} X_i\lra (A\otimes D(S_i))\otimes_{\m}X\lra 0.$$
Clearly, $X$ is a monic representation if and only if the above sequence is exact for all $i\in Q_0$ , if and only if $$\Tor_m^{\m}(A\otimes D(S_i), X)=0$$
 for all $i\in Q_0$ and $m\ge1$. This shows that the property $X$ is a monic representation can be characterized homologically by the above vanishing property on $\Tor$-groups.   For this reason, we define {\em monic representations} of an arbitrary finite dimensional algebra $B$ over another algebra $A$ as follows.

\begin{defn}\label{Def: homological monic representations}
Let $A$ and $B$ be finite dimensional $k$-algebras, and let $\m:=A\otimes B$. A left $\m$-module $X$ is called a {\bf monic representation} of $B$ over $A$ if $$\Tor_i^{\m}(A\otimes \D(S), X)=0$$
for all $i\ge1$ and all simple left $B$-modules $S$. We denote by $\Mon(B, A)$ the full subcategory of $\modcat{\Lambda}$ consisting of all  monic representations of $B$ over $A$, called the {\bf monomorphism category} of $B$ over $A$.
\end{defn}

For each simple left $B$-module $S$, let $\cdots\ra P^{-m}\ra\cdots\ra P^0\ra \D(S)\ra 0$ be a projective resolution of the simple right $B$-module $\D(S)$. Then $\cdots\ra A\otimes P^{-m}\ra \cdots\ra A\otimes P^0\ra A\otimes \D(S)\ra 0$ is a projective resolution of the right $\Lambda$-module $A\otimes \D(S)$.  By definition,  a left $\Lambda$-module $X$ belongs to $\Mon(B, A)$ if and only if the sequence of left $A$-modules
$$\cdots\lra (A\otimes P^{-m})\otimes_{\Lambda}X\lra\cdots\lra (A\otimes P^0)\otimes_{\Lambda}X\lra (A\otimes \D(S))\otimes_{\Lambda}X\lra  0$$
is exact for all  simple left $B$-modules $S$.

\medskip

Keep the notations above.  The following lemma characterizes modules in $\Mon(B, A)$.

\begin{lem} \label{lemma-monicRep-equiv-condition}
Let $X$ be a left $\m$-module. The following are equivalent:
\begin{itemize}
   \setlength\itemsep{0mm}
 \item[$(1)$.] $X\in\Mon(B, A)$;

 \item[$(2)$.] ${\rm Tor}_i^{\m}(A\otimes  V, X)=0,$ for all $i\geq 1$ and for all $V\in $ {\rm mod-}$B$;

 \item[$(3)$.] ${\rm Ext}^i_{\m}(X, \D(A)\otimes  S)=0$, for all $i\geq 1$ and for all simple left $B$-modules $S$;

 \item[$(4)$.] ${\rm Ext}^i_{\m}(X, \D(A)\otimes  Y)=0$, for all $i\geq 1$ and for all $Y\in \modcat{B}$.

 \item[$(5)$.] ${}_BX$ is a projective $B$-module
\end{itemize}

\noindent If $\gldim(B)<\infty$, then the above statements are further equivalent to the following two statements:
\begin{itemize}
 \setlength\itemsep{0mm}
 \item[$(6)$.] $\Tor_i^{\m}(A\otimes  \D(B), X)=0$ for all $i\geq 1$.

 \item[$(7)$.] $\Ext^i_{\m}(X, \D(A)\otimes  B)=0$ for all $i\geq 1$.
\end{itemize}
\end{lem}

\begin{proof}
Note that $X$ is a monic representation of $B$ over $A$ if and only if $\Tor^{\m}_i(A\otimes  \D(S), X)=0$ for all $i\geq 1$ and for all simple left $B$-modules $S$. Since any right $B$-module in $\mbox{\rm mod-}B$ has finite length,  the statement (1) and (2) are equivalent. Similarly (3) and (4) are equivalent.

The fact (1) and (3) are equivalent follows directly from the well-known isomorphism $\D\Tor_i^{\m}(A\otimes D(S), X)\cong \Ext_{\m}^i(X, \D(A)\otimes S)$, which in turn is a consequence of Lemma~\ref{lemma-basic-homological-isomorphism} (1) and Lemma~\ref{lemma-tensorModuleHom} (2). By Lemma \ref{lemma-basic-homological-isomorphism} (3), there is a  natural isomorphism
$$\Ext_B^i({}_BX, {}_BY)\cong\Ext_B^i({}_B\Lambda\otimes_{\Lambda}X, {}_BY)\cong\Ext_{\Lambda}^i(X, \Hom_B(A\otimes B, Y))\cong\Ext_{\Lambda}^i(X, \D(A)\otimes Y)$$
for all $i\geq 0$. Note that ${}_BX$ is projective if and only if $\Ext_B^i({}_BX, {}_BY)=0$ for all $Y\in\modcat{B}$ and all $i>0$. The above isomorphism indicates that $(4)$ and $(5)$ are equivalent.  Thus (1) -- (5) are all equivalent.

Suppose that $\gldim(B)<\infty$.  Clearly $(2)$ implies $(6)$.   Assume now that $(6)$ holds. By assumption, every right $B$-module $V$ has finite injective resolution $0\ra V\ra I_0\ra\cdots\ra I_m\ra 0$. Applying $A\otimes -$ to this sequence results in an exact sequence $0\ra A\otimes V\ra A\otimes I_0\ra\cdots\ra A\otimes I_m\ra 0$. Since $\Tor_i^{\Lambda}(A\otimes I_s, X)=0$ for all $i\geq 1$ and $0\leq s\leq m$ by (6), we deduce that $\Tor_i^{\Lambda}(A\otimes V, X)=0$ for all $i\geq 1$. This proves that $(6)\Rightarrow (2)$.  That
$(6)$ is equivalent to $(7)$ again follows from the natural isomorphism $\D\Tor_i^{\Lambda}(A\otimes \D(B), X)\cong \Ext_{\Lambda}^i(X, \D(A)\otimes B)$.
\end{proof}

\begin{exm}
\label{example-not-symmetric}
For two algebras $A$ and $B$, the definition of $\Mon(B, A)$ is not symmetric, that is, the categories $\Mon(B, A)$ and $\Mon(A, B)$ do not coincide in general. For instance,  let $A=k[x]/(x^2)$ and $B$ be the path algebra of the quiver $1\lraf{\alpha}2$ over $k$. Then $\Lambda:=A\otimes B$ is given by the following quiver
$$\xymatrix{\bullet\ar[r]^{\alpha}_(0){1}_(1){2}\ar@(dl,ul)^{\beta} & \bullet\ar@(dr,ur)_{\gamma}}$$
with relations $\beta^2, \gamma^2, \alpha\beta-\gamma\alpha$.
Note that there are 9 indecomposable modules in $\m$-mod up to isomorphism. Let $S$ be the unique simple $A$-module, and denote by $\lambda: S\ra A$ the inclusion map and by $\pi: A\ra S$ the canonical surjective map. By calculation, one can see that $\Mon(B, A)$ has 5 indecomposable objects (written as representation of $B$ over $A$):
$$0\lra S, \quad 0\lra A, \quad S\lraf{1} S, \quad  S\lraf{\lambda} A, \quad A\lraf{1} A$$
while $\Mon(A, B)$ has 4 indecomposable objects: $0\lra A, \quad A\lra 0, \quad A\lraf{1}A,\quad  A\lraf{\pi\lambda}A$.
\end{exm}

Recall that, for an algebra $\Gamma$, a full subcategory of $\modcat{\Gamma}$ is called a resolving subcategory provided that it contains all projective $\Gamma$-modules and is closed under taking extensions, direct summands, and kernels of epimorphisms.  Typical resolving subcategory includes  ${}^{\perp}\mathscr{X}$ which is, for each  full subcategory $\mathscr{X}$ of $\modcat{\Gamma}$, defined to be
 $${}^{\perp}\mathscr{X}:=\{Y\in\modcat{\Gamma} \ |\ \Ext_{\Gamma}^i(Y, \mathscr{X})=0\mbox{ for all }i\geq 1\}.$$
In case that $\mathscr{X}=\{X\}$ consists of only one $\Gamma$-module, we write ${}^{\perp}X$ for ${}^{\perp}\mathscr{X}$.

\medskip
There is an immediate consequence of Lemma \ref{lemma-monicRep-equiv-condition}.

\begin{cor}\label{corollary-Mon(BA)-resolving}
$\Mon(B, A)$ is a resolving subcategory of $\modcat{\Lambda}$. Moreover, $\Mon(B, A)$ is a Krull-Schmidt exact category.
\end{cor}

\begin{proof}
Let $\mathscr{X}:=\{\D(A)\otimes V| V\in\modcat{B}\}$. Then,  by Lemma \ref{lemma-monicRep-equiv-condition} (4), the monomorphism category $\Mon(B, A)$ coincides with  ${}^{\perp}\mathscr{X}$, which is a resolving subcategory of $\modcat{\Lambda}$.
\end{proof}

The following lemma will be very useful in later proofs.

\begin{lem}\label{lemma-monic-Ext-change-ring}
Let $A$ and $B$ be two finite dimensional $k$-algebras. Suppose that ${}_{\Lambda}X\in \Mon(B, A)$. Let $V_B$ be a right $B$-module, and let ${}_AU$ be a left $A$-module. For each $i\geq 0$, there is a natural isomorphism
$$\Ext^i_{A}((A\otimes V)\otimes_{\Lambda}X, {}_AU)\cong \Ext_{\Lambda}^i(X, U\otimes \D(V)).$$
\end{lem}

\begin{proof}
Since ${}_{\Lambda}X\in\Mon(B, A)$, we have $\Tor_i^{\Lambda}(A\otimes V, X)=0$ for all $i>0$. Note that $\Ext_A^i(A\otimes V, U)=0$ for all $i>0$ since ${}_A(A\otimes V)$ is projective. Thus, it follows from Lemma \ref{lemma-basic-homological-isomorphism} (3) that
$$\Ext_{A}^i((A\otimes V)\otimes_{\Lambda}X, U)\cong \Ext_{\Lambda}^i(X, \Hom_{A}(A\otimes V, U))\cong \Ext_{\Lambda}^i(X, U\otimes D(V))$$
for all $i\geq 0$.
\end{proof}

\subsection{Properties}
In this subsection, we keep the notations above. In general, we can define the monomorphism category of $B$ over a full additive subcategory $\mathscr{L}$ of $\modcat{A}$, namely $\Mon(B, \mathscr{L})$, which is a full subcategory of $\Mon(B, A)$. Note that the $\Lambda$-module in the  monomorphism category  $\Mon(B, A)$ is not Gorenstein projective in general. To classify  Gorenstein projective $\Lambda$-modules in $\Mon(B, A)$, we need the special monomorphism category $\Mon(B, A\Gp)$.

In this subsection, we shall study some properties of $\Mon(B, \mathscr{L})$. At first, the precise definition is as follows.

 \begin{defn}
 Let $\mathscr{L}$ be a full additive   subcategory of $\modcat{A}$.  Define
$$\Mon(B, \mathscr{L}):=\{X\in\Mon(A, B)\ |\ (A\otimes V)\otimes_{\Lambda}X\in \mathscr{L}\mbox{ whenever } V\in \mbox{\rm mod-}B\}.$$
\end{defn}

If $\mathscr{L}=\modcat{A}$, then $\Mon(B, \mathscr{L})$ is just $\Mon(B, A)$ itself.

\medskip

 The following proposition collects some facts on $\Mon(B, \mathscr{L})$.

\begin{prop} \label{proposition-mon(LB)}
Let $\mathscr{L}$ be an additive full subcategory of $\modcat{A}$. We have the following statements.
\begin{itemize}
 \setlength\itemsep{0mm}
 \item[$(1)$.] Suppose that $\mathscr{L}$ is closed under extensions, and that $X\in\Mon(B, A)$. Then the following are equivalent.
\begin{itemize}
 \setlength\itemsep{0mm}
 \item[$(a)$.] $X\in\Mon(B, \mathscr{L})$.

 \item[$(b)$.] The $A$-module  $(A\otimes V)\otimes_{\Lambda}X$ belongs to $\mathscr{L}$ for all projective and for all simple right $B$-modules $V$.

 \item[$(c)$.]  For each simple right $B$-module  $S$, the $A$-module  $(A\otimes S)\otimes_{\Lambda}X$ belongs to $\mathscr{L}$.
\end{itemize}

\item[$(2)$.] Assume that $\mathscr{L}$ is closed under direct summands. Then
\begin{itemize}
 \setlength\itemsep{0mm}
 \item[$(a)$.] If $X\in\Mon(B, \mathscr{L})$, then for any projective left $B$-module $P$,  ~${}_A\Hom_B(P, {}_BX)\in \mathscr{L}$. In particular, ${}_AX\in \mathscr{L}$.

 \item[$(b)$.] If $0\neq L\in \mathscr{L}$ and $U\in\modcat{B}$, then $L\otimes U\in \Mon(B, \mathscr{L})$ if and only if ${}_BU$ is projective .
\end{itemize}
 \setlength\itemsep{0mm}
 \item[$(3)$.] If $\mathscr{L}$ is closed under taking extensions (resp. kernels of epimorphisms, direct summands), then so is ${\rm Mon}(B, \mathscr{L})$.

 \item[$(4)$.] If $\mathscr{L}$ is a resolving subcategory of $\modcat{A}$, then  $\Mon(B, \mathscr{L})$ is a resolving subcategory of $\modcat{\Lambda}$.
\end{itemize}
\end{prop}

\begin{proof}
$(1).$ The implications $(a)\Rightarrow (b)$ and $(b)\Rightarrow (c)$ are obvious. Since $X\in\Mon(B, A)$, we have $\Tor_i^{\Lambda}(A\otimes V, X)=0$ for all $V\in\mbox{\rm mod-}B$ and $i\ge1$. Thus $(A\otimes -)\otimes_{\Lambda}X$ is an exact functor from {\sf mod}-$B$ to $\modcat{A}$.  Since $\mathscr{L}$ is closed under extensions and $(A\otimes S)\otimes_{\Lambda}X\in\mathscr{L}$ for all simple right $B$-modules $S$, we deduce that $(A\otimes V)\otimes_{\Lambda}X$ belongs to $\mathscr{L}$ for all right $B$-modules $V$. This proves that $(c)\Rightarrow (a)$.

\smallskip

$(2.a).$  Note that, for each $X\in\Mon(B, \mathscr{L})$ and for each indecomposable projective left $B$-module $P$, the $A$-module  ${}_A\Hom_B(P, {}_BX)$ is a direct summand of ${}_A\Hom_B(B, {}_BX)\cong {}_AX\cong {}_A(A\otimes B)\otimes_{\Lambda}X$ which is in $\mathscr{L}$.
 Hence $\Hom_B(P, X)$ belongs to $\mathscr{L}$ as a left $A$-module for all projective left $B$-module $P$ and ${}_AX\in\mathscr{L}$.

\smallskip

$(2.b).$ By Lemma \ref{lemma-monicRep-equiv-condition} , the $\Lambda$-module $L\otimes U\in\Mon(B, A)$ if and only ${}_B(L\otimes U)$ is projective if and only if ${}_BU$ is projective.   Moreover, in this case, $(A\otimes V)\otimes_{\Lambda}(L\otimes U)\cong (A\otimes_AL)\otimes (V\otimes_BU)\in\add(L)$  as left $A$-modules for all right $B$-modules $V$. It follows that $L\otimes U$ is actually in $\Mon(B, \mathscr{L})$.

\smallskip

$(3)$.  Suppose that $0\to X\to Y \to Z\to 0$ is a short exact sequence in $\modcat{\Lambda}$ with $X, Z\in\Mon(B, \mathscr{L})$. Let $V$ be a right $B$-module. Applying  $(A\otimes V)\otimes_{\Lambda}-$ results in an exact sequence
$$\Tor_i^{\Lambda}(A\otimes V, X)\lra \Tor_i^{\Lambda}(A\otimes V, Y)\lra \Tor_i^{\Lambda}(A\otimes V, Z) $$
for all $i\geq 0$.  By Lemma \ref{lemma-monicRep-equiv-condition}, the first and the third terms vanish for all $i\geq 1$.  Hence the second term $\Tor_i^{\Lambda}(A\otimes V, Y)$ vanishes for all $i\geq 1$.  Consequently, $Y\in\Mon(B, A)$ by Lemma \ref{lemma-monicRep-equiv-condition}. Moreover, there is a short exact sequence
$$0\lra (A\otimes V)\otimes_{\Lambda}X\lra (A\otimes V)\otimes_{\Lambda}Y\lra (A\otimes V)\otimes_{\Lambda}Z\lra 0$$
of left $A$-modules. Note that the terms $(A\otimes V)\otimes_{\Lambda}X$ and $(A\otimes V)\otimes_{\Lambda}Z$ belong to $\mathscr{L}$ as left $A$-modules, since $X, Z\in\Mon(B, \mathscr{L})$.  Since $\mathscr{L}$ is closed under extensions, we deduce that $(A\otimes V)\otimes_{\Lambda}Y\in \mathscr{L}$,  and therefore $Y\in\Mon(B, \mathscr{L})$.  In case that $\mathscr{L}$ is closed under kernels of epimorphisms or direct summands, one can similarly show that so is $\Mon(B, \mathscr{L})$.

\smallskip

$(4)$. By assumption, $\mathscr{L}$ contains $\projcat{A}$ and is closed under taking extensions, direct summands and kernels of epimorphisms. By $(2.b)$ and $(3)$, we see that $\Mon(B, \mathscr{L})$ contains $A\otimes B$ and is closed under extensions, direct summands and kernels of epimorphisms. Namely, $\Mon(B, \mathscr{L})$ is a resolving subcategory of $\modcat{\Lambda}$.
\end{proof}

Let $\Gamma$ be an algebra, and let $\mathscr{X}$ be a full subcategory of $\modcat{\Gamma}$ closed under extensions and direct summands.  An object $U$ in $\mathscr{X}$ is said to be a {\em projective object in $\mathscr{X}$} provided that, for each short exact sequence $0\ra X\ra Y\ra Z\ra 0$ in $\mathscr{X}$, the sequence $0\ra \Hom_{\Gamma}(U, X)\ra \Hom_{\Gamma}(U, Y)\ra \Hom_{\Gamma}(U, Z)\ra 0$ is again exact.  Similarly, one can define injective objects in $\mathscr{X}$.  In general, projective objects or injective objects in $\mathscr{X}$ may not exist.  An object $I\in\mathscr{X}$ is called an {\em injective cogenerator} of $\mathscr{X}$ provided that $I$ is an injective object in $\mathscr{X}$ and that, for each $X\in\mathscr{X}$, there is a short exact sequence $0\ra X\ra I_X\ra Y\ra 0$ in $\mathscr{X}$ with $I_X\in\add(I)$. The {\em projective generator} in $\mathscr{X}$ is defined dually.

We consider projective and  injective  objects in $\mathrm{Mon}(B, \mathscr{L})$.

\begin{prop}\label{proposition-proj-in-Mon(BL)}
Let $\mathscr{L}$ be a  full subcategory of $\modcat{A}$ closed under extensions and direct summands.
\begin{itemize}
 \setlength\itemsep{0mm}
\item[$(1).$]  If $L$ is a projective  object in $\mathscr{L}$  and $P$ a projective left $B$-module, then $L\otimes P$ is a projective  object in $\Mon(B, \mathscr{L})$.

\item[$(2). $]  Assume that for each $X\in \Mon(B, \mathscr{L})$, the kernel of the natural $A$-module homomorphism
 $$X\otimes  (B/\rad(B))\to (B/\rad(B))\otimes_B X, \quad  x\otimes b\mapsto b\otimes x$$ belongs to $\mathscr{L}$. In particular, this holds when $\mathscr{L}$ is closed under kernels of epimorphisms. Then we have the following statements.
\begin{itemize}
	\item[{\rm (a).}]  If $\mathscr{L}$ has enough projective objects,   so does $\mathrm{Mon}(B, \mathscr{L})$.

    \item[{\rm (b).}]  Let $P$ be a projective generator of $\mathscr{L}$. Then $P\otimes  B$ is a projective generator of $\mathrm{Mon}(B, \mathscr{L})$.\end{itemize}
\end{itemize}
\end{prop}

\begin{proof}

(1). By Proposition~\ref{proposition-mon(LB)} (3), $\Mon(B, \mathscr{L})$ is an  extension-closed subcategory of $\Lambda\modcat$, is thus an exact category.
Suppose that $L$ is a projective object in $\mathscr{L}$ and that $P$ is a projective left $B$-module. By Lemma~\ref{lemma-tensorModuleHom} (6), there is a natural isomorphism
$$\Hom_\Lambda(L\otimes  P, -)\cong \Hom_A (L, \Hom_B(P, -))$$
and by Proposition~\ref{proposition-mon(LB)} (2.a) , $ \Hom_B(P, -)$ is a well defined  exact functor from  $\Mon(B, \mathscr{L})$  to  $\mathscr{L}$. This shows that the functor
 $ \Hom_\Lambda(L\otimes  P, -): \Mon(B, \mathscr{L})\to k\modcat$ is exact.
  Hence $L\otimes P$ is a projective object in $\Mon(B, \mathscr{L})$.

\vskip5pt

$(2)$. Since $\mathscr{L}$ is closed under extensions and direct summands, the supplementary hypothesis is equivalent to saying that  for  each $X\in \Mon(B, \mathscr{L})$ and each simple right $B$-module $S$, the kernel of the natural $A$-module homomorphism
$$X\otimes  S\to S\otimes_B X, x\otimes b\mapsto b\otimes x$$
belongs to $\mathscr{L}$ and the latter  holds if and only if for each right $B$-module $V$, the kernel of the natural $A$-module homomorphism
  $$f: X\otimes  V\to V\otimes_B X, x\otimes v\mapsto v\otimes x$$
 belongs to $\mathscr{L}$.

 Under this supplementary hypothesis, we  claim that for  any  $X\in \Mon(B, \mathscr{L})$   the kernel of the natural homomorphism of $\Lambda$-modules $g: X\otimes  B\to X, \ x\otimes b\mapsto bx$ lies in $\Mon(B,
\mathscr{L})$. In fact, since $X\in \Mon(B, \mathscr{L})$ , for each right $B$-module $V$, applying the functor $(A\otimes  V)\otimes_\Lambda -\cong V\otimes_B-$ to the short exact sequence
$$0\to \mathrm{Ker}(g)\to X\otimes  B\stackrel{g}{\to} X\to 0,$$
gives a short exact sequence of left $A$-modules
$$0\to V\otimes_B\mathrm{Ker}(g)\to V\otimes_B(X\otimes  B)\cong  X\otimes  V \stackrel{f}{\to} V\otimes_BX\to 0. $$
The hypothesis shows that $(A\otimes V)\otimes_{\Lambda}\mathrm{Ker}(g)\cong V\otimes_B\mathrm{Ker}(g)\in \mathscr{L}$ and as a consequence, $\mathrm{Ker}(g)\in \Mon(B, \mathscr{L})$.

Let $X\in \Mon(B, \mathscr{L})$. Since $\mathscr{L}$ has enough projective objects, there exists an epimorphism $h: L\to {}_A X$ with $L$ a projective object in $\mathscr{L}$ and  $\mathrm{Ker}(h)\in \mathscr{L}$.
 Then the composite $L\otimes  B \stackrel{h\otimes id_B}{\to} X\otimes  B\stackrel{g}{\to} X$ is an epimorphism.  By (1), $L\otimes  B$ is a projective  object in $\Mon(B, \mathscr{L})$. It suffices to show that the kernel of this composite belongs to   $\Mon(B, \mathscr{L})$.  In fact, this kernel is an extension of $\mathrm{Ker}(g)$ by $\mathrm{Ker}(h\otimes id_B)\cong \mathrm{Ker}(h)\otimes  B$. The module $\mathrm{Ker}(g)\in \Mon(B, \mathscr{L})$ by the previous paragraph and  $\mathrm{Ker}(h\otimes id_B) \in \Mon(B, \mathscr{L})$ follows from Proposition~\ref{proposition-mon(LB)} (2.b). Since $\Mon(B, \mathscr{L})$ is closed under extensions,  the kernel of this composite belongs to   $\Mon(B, \mathscr{L})$. This proves (a).

Let $P$ be a projective generator of $\mathscr{L}$. In the  previous paragraph, one can take $L\in \mathrm{add}(P)$. Then the starting  term $L\otimes  B$ of the epimorphism  $L\otimes  B \stackrel{h\otimes id_B}{\to} X\otimes  B\stackrel{g}{\to} X$
belongs to $\mathrm{add}(P\otimes  B)$. This shows that $P\otimes  B$ is a projective generator of $\mathrm{Mon}(B, \mathscr{L})$.
\end{proof}

\begin{prop}\label{proposition-inj-in-Mon(BL)}
Let $\mathscr{L}$ be a  full subcategory of $\modcat{A}$ closed under extensions and direct summands.
\begin{itemize}
 \setlength\itemsep{0mm}
\item[$(1)$.]  If $L$ is an injective object in $\mathscr{L}$ and $P$ a projective left $B$-module, then $L\otimes P$ is an injective object in $\Mon(B, \mathscr{L})$.

\item[$(2)$.]  If $\mathscr{L}$ has enough injective objects, then so does $\Mon(B, \mathscr{L})$.

\item[$(3)$.] Let $I$ be an injective cogenerator of $\mathscr{L}$. Then $I\otimes  B$ is an injective cogenerator of $\Mon(B, \mathscr{L})$.
\end{itemize}
\end{prop}

 \begin{proof}
 (1). Using the natural isomorphism
$$\Hom_{A}((A\otimes \D(P))\otimes_{\Lambda}X, L)\cong \Hom_{\Lambda}(X, L\otimes P)$$
provided by Lemma \ref{lemma-tensorModuleHom} (5), we see that $L\otimes P$ is an injective object in $\Mon(B, \mathscr{L})$, as the functor $(A\otimes \D(P))\otimes_\Lambda-$ is an exact functor from $\Mon(B, \mathscr{L})$ to $\mathscr{L}$ and  $L$ is an injective object in $\mathscr{L}$.

\vskip5pt

 (2). Let $X\in\Mon(B, \mathscr{L})$.  Then $(A\otimes V)\otimes_{\Lambda}X\in\mathscr{L}$ for all right $B$-modules $V$. Since $\mathscr{L}$ has enough injective objects, there is an exact sequence
$$0\lra (A\otimes \D(B))\otimes_{\Lambda}X\lra I\lra L'\lra 0$$
in $\mathscr{L}$ such that $I$ is an injective object in $\mathscr{L}$.  We fix this $I$ for the rest of the proof.

\medskip
\noindent{\em Claim 1: For each right $B$-module $V$,  there is an embedding $i_V:   (A\otimes V)\otimes_{\Lambda}X\lra I_V$ with $I_V\in\add(I)$ such that $\cok (i_V)\in\mathscr{L}$. }

\medskip
 Actually, let $ 0\ra V\ra E\ra V'\ra 0$ be an exact sequence in $\mbox{\rm mod-}B$ with $E$ injective. Applying  $(A\otimes -)\otimes_{\Lambda}X$ gives rise to an exact sequence
 $$0\lra (A\otimes V)\otimes_{\Lambda}X\lraf{t} (A\otimes E)\otimes_{\Lambda}X\lra (A\otimes V')\otimes_{\Lambda}X\lra 0.$$
Since $E$ is injective, there is an embedding $s: (A\otimes E)\otimes_{\Lambda}X\ra I_V$ for some $I_V\in\add(I)$ such that $\cok(s)\in\mathscr{L}$. Define $i_V$ to be the composite $st$. Then $i_V$ is monic and $\cok(i_V)$, which is an extension of  $\cok(t)$ and $\cok(s)$, belongs to $\mathscr{L}$.  This proves Claim 1.

\medskip
 Let $\eta: X\lra M$ be a left $\add(I\otimes B)$-approximation of $X$. By (1), the module $M$ is an injective object in $\Mon(B, \mathscr{L})$. It suffices to prove that $\eta$ is a monomorphism and $\cok(\eta)$ still lies in $\Mon(B, \mathscr{L})$.

\medskip
\noindent
{\em Claim 2. For each ${}_BW\in\modcat{B}$ and $I'\in\add(I)$, every morphism in $\Hom_{\Lambda}(X, I'\otimes W)$ factors through $\eta$.}

 \medskip
 We first show that $\Ext_{\Lambda}^1(X, I'\otimes U)=0$ for all ${}_BU\in\modcat{B}$.  Actually,  by Lemma \ref{lemma-monic-Ext-change-ring}, there is an isomorphism $$\Ext_{\Lambda}^1(X, I'\otimes U)\cong \Ext_{A}^1((A\otimes D(U))\otimes_{\Lambda}X, I').$$
Moreover, the Ext-group $\Ext_{A}^1((A\otimes D(U))\otimes_{\Lambda}X, I')$ vanishes since $\mathscr{L}$ is closed under extensions and $I'$ is an injective object in $\mathscr{L}$. Hence $\Ext_{\Lambda}^1(X, I'\otimes U)=0$ for all ${}_BU\in\modcat{B}$.

Let $W$ be an arbitrary module in $\modcat{B}$, and let $g: X\ra I'\otimes W$ be an arbitrary $\Lambda$-module homomorphism. We can form the following diagram,
$$\xymatrix@R=12mm{
&& X\ar[rd]^(.35){g}\ar@{-->}[d]_{\alpha}\ar[r]^{\eta} &M\ar@{-->}[ld]^(.25){\beta}\\
0\ar[r] & I'\otimes  U\ar[r] & I'\otimes P_W\ar[r]^{\pi} &I'\otimes W\ar[r] & 0
}$$
such that $g=\pi\alpha$ and $\alpha=\beta \eta$, where $P_W$ is a projective cover of $_BW$. The fact $\Ext_{\Lambda}^1(X, I'\otimes  U)=0$  guarantees the existence of $\alpha$. The existence of $\beta$ follows from the fact that $\eta$ is a left $\add(I\otimes B)$-approximation. Hence $g=\pi\beta \eta$ factors through $\eta$.

\medskip
\noindent
{\em Claim 3. The morphism $\eta$ is a monomorphism.}

\medskip
By Claim 1, there is an embedding $i_B: {}_AX={}_A(A\otimes B)\otimes_{\Lambda}X\ra I_B$ for some $I_B\in\add(I)$ such that $\cok(i_B)\in\mathscr{L}$.
 Let $\theta: {}_{\Lambda}X\ra \Hom_A(\Lambda, I_B)$ be the image of $i_B$ under the natural isomorphism $$\Hom_{A}({}_AX, I_B)\cong\Hom_{\Lambda}(X, \Hom_A(\Lambda, I_B)).$$
By definition, the map $\theta$ sends each $x\in X$ to $(r\mapsto i_B(rx))$. It follows that $\Ker(\theta)=0$. Actually, for each $x\in\Ker(\theta)$, one has $i_B(rx)=0$ for all $r\in \Lambda$. Particularly, $i_B(1\cdot x)=0$, and hence $x=0$ since $i_B$ is a monomorphism. Note that there is an isomorphism $\Hom_A(\Lambda, I_B)\cong I_B\otimes D(B)$. From Claim 2, we deduce that $\theta$ factors through $\eta$. As a result, the morphism $\eta$ must be a monomorphism. This proves Claim 3.

\medskip
\noindent
{\em Claim 4: The cokernel of $\eta$ belongs to $\Mon(B, \mathscr{L})$}.

\medskip
Let $Z:=\cok(\eta)$, and let $V$ be an arbitrary right $B$-module. Applying $(A\otimes V)\otimes_{\Lambda}-$ to the exact  sequence
  $$0\lra X\lraf{\eta} M\lra Z\lra 0,$$ one gets an exact sequence in $\modcat{A}$
$$ (A\otimes V)\otimes_{\Lambda}X\lraf{1\otimes \eta}(A\otimes V)\otimes_{\Lambda}M\lra (A\otimes V)\otimes_{\Lambda}Z\lra 0. \quad\quad (\star)$$
By Claim 1, there is an embedding $i_V: (A\otimes V)\otimes_{\Lambda}X\ra I_V$ with $I_V\in\add(I)$ such that $\cok(i_V)\in\mathscr{L}$. Applying   $\Hom_A(-, I_V)$ to $(\star)$ gives rise to a commutative diagram with exact rows.
 $$\xymatrix@M=2mm{
 0\ar[r] &\Hom_{A}((A\otimes V)\otimes_{\Lambda}Z, I_V)\ar[d]^{\cong}\ar[r] & \Hom_{A}((A\otimes V)\otimes_{\Lambda}M, I_V)\ar[r]^{(1\otimes \eta)^*}\ar[d]^{\cong} & \Hom_{A}((A\otimes V)\otimes_{\Lambda}X, I_V)\ar[d]^{\cong}\\
  &\Hom_{\Lambda}(Z, I_V\otimes D(V))\ar[r] & \Hom_{\Lambda}(M, I_V\otimes D(V))\ar[r]^{\eta^*} & \Hom_{\Lambda}(X, I_V\otimes D(V))\ar[r] &0.\\
 }$$
 By Claim 2, the map  $\eta^*$ is surjective. It follows that $(1\otimes \eta)^*$ is surjective, and particularly $i_V$ factors through $1\otimes\eta$. Hence $1\otimes\eta$ must be a monomorphism.  Note that both $X$ and  $M$ also belong to $\Mon(B, \mathscr{L})$. By Lemma \ref{lemma-monicRep-equiv-condition}, we have $\Tor_m^{\Lambda}(A\otimes V, X)=0=\Tor_m^{\Lambda}(A\otimes V, M)$ for all $m\ge1$.  Together with the fact $1\otimes\eta$ is injective, we deduce that   $\Tor_m^{\Lambda}(A\otimes V, Z)=0$ for all $m\geq 1$, that is, $Z\in\Mon(B, A)$.  Moreover, we can form the following commutative diagram.
$$\xymatrix{
0\ar[r] & (A\otimes V)\otimes_{\Lambda}X\ar[r]^{1\otimes \eta}\ar@{=}[d] & (A\otimes V)\otimes_{\Lambda}M\ar[r]\ar@{-->}[d] & (A\otimes V)\otimes_{\Lambda}Z\ar[r]\ar@{-->}[d] & 0\\
0\ar[r] & (A\otimes V)\otimes_{\Lambda}X\ar[r]^{i_V} & I_V \ar[r] & \cok(i_V)\ar[r] & 0
}$$
in $\modcat{A}$ with exact rows.
The right square is a pullback and a pushout. Thus, we obtain a short exact sequence
$$0\lra (A\otimes V)\otimes_{\Lambda}M\lra I_V\oplus \big((A\otimes V)\otimes_{\Lambda}Z\big)\lra \cok(i_V)\lra 0$$
in $\modcat{A}$.  Since $\cok(i_V)$,  $(A\otimes V)\otimes_{\Lambda}M$ and $I_V$ are all in $\mathscr{L}$,  and   $\mathscr{L}$ is closed under extensions and direct summands, we deduce that $(A\otimes V)\otimes_{\Lambda}Z\in\mathscr{L}$. Hence $Z\in\Mon(B, \mathscr{L})$.

\vskip5pt

(3). The proof of (2) actually shows that if there is an embedding of $(A\otimes D(B))\otimes_{\Lambda}X$ into an injective object $I\in \mathscr{L}$ with cokernel in $\mathscr{L}$, then ${}_{\Lambda}X$ can be embedded into a module in $\add(I\otimes B)$ with cokernel in $\Mon(B, \mathscr{L})$. Thus (3) follows.
 \end{proof}

Let us remark that a special case of Proposition \ref{proposition-inj-in-Mon(BL)}, where $B$ is a path algebra of a finite acyclic quiver and $\mathscr{L}=\modcat{A}$, was studied in \cite[Theorem 1]{Song2016} by using combinatoric methods.

\begin{prop} \label{proposition-Frobenius}
The category $\Mon(B, A\Gp)$  is a Frobenius exact category with the projective-injective objects being projective $\Lambda$-modules. Moreover, the following are equivalent:
\begin{itemize}
 \setlength\itemsep{0mm}
 \item[$(1)$.] The exact category $\Mon(B, A)$ is Frobenius;

 \item[$(2)$.] $A$ is a selfinjective algebra;

 \item[$(3)$.] $\Mon(B, A)=\Mon(B, A\Gp)$.
\end{itemize}
\end{prop}

\begin{proof}This follows from Propositions~\ref{proposition-proj-in-Mon(BL)} and \ref{proposition-inj-in-Mon(BL)}, as $ A\Gp$ is a Frobenius category and is resolving in $A\modcat$.
\end{proof}

\medskip

The following proposition shows that how bimodules can transfer monic representations over one algebra to another.

\begin{prop}\label{proposition-Mon(BL)-to-Mon(CL)}
Keep the notations above.
Let $C$ be another algebra, and let ${}_CM_B$ be a $C$-$B$-bimodule. Suppose that $\mathscr{L}$ is a full subcategory of $\modcat{A}$ closed under extensions.  If ${}_CM$ is projective, then $(A\otimes M)\otimes_{\Lambda}-$ induces a functor from $\Mon(B, \mathscr{L})$ to $\Mon(C, \mathscr{L})$.
\end{prop}

 \begin{proof}
 Set $\Lambda':=A\otimes C$.
 Let $X$ be a module in $\Mon(B, \mathscr{L})$. Then $\Tor_i^{\Lambda}(A\otimes M, X)=0$ for all $i\ge1$. Since ${}_CM$ is projective, the left $\Lambda'$-module $A\otimes M$ is projective, and thus $\Tor_i^{\Lambda'}(A\otimes V, A\otimes M)=0$ for all $i\ge1$ and for all right $C$-modules $V$. Then, by Lemma \ref{lemma-basic-homological-isomorphism} (2), we get
 $$\Tor_i^{\Lambda'}(A\otimes V, (A\otimes M)\otimes_{\Lambda}X)\cong \Tor_i^{\Lambda}((A\otimes V)\otimes_{\Lambda'}(A\otimes M), X)$$
 for all $i\geq 0$, where $V$ is an arbitrary right $C$-module. Note that $\Tor_i^{\Lambda}((A\otimes V)\otimes_{\Lambda'}(A\otimes M), X)\cong\Tor_i^{\Lambda}(A\otimes  (V\otimes_CM), X)=0$ for all $i\ge1$. It follows that $\Tor_i^{\Lambda'}(A\otimes V, (A\otimes M)\otimes_{\Lambda}X)=0$ for all $i\ge1$. Hence $(A\otimes M)\otimes_{\Lambda}X\in\Mon(C, A)$. The associativity of tensor product gives that $(A\otimes V)\otimes_{\Lambda'}((A\otimes M)\otimes_{\Lambda}X)\cong (A\otimes (V\otimes_CM))\otimes_{\Lambda}X$ which is in $\mathscr{L}$ by assumption. Hence $(A\otimes M)\otimes_{\Lambda}X\in\Mon(C, \mathscr{L})$.
  \end{proof}

As a corollary, Morita equivalent algebras have equivalent monomorphism categories.

\begin{cor}\label{corollary-Morita-BL-CL}
Suppose that the bimodule ${}_CM_B$ induces a Morita equivalence between two algebras $C$ and $B$. Let $\mathscr{L}$ be a full subcategory of $\modcat{A}$ closed under extensions. Then the functor $(A\otimes M)\otimes_{\Lambda}-$ induces an equivalence between $\Mon(B, \mathscr{L})$ and $\Mon(C, \mathscr{L})$.
\end{cor}

\section{Gorenstein projective bimodules via monomophism categories}
 Throughout this section, we fix two finite dimensional $k$-algebras $A$ and $B$, and set $\Lambda:=A\otimes B$ to be their tensor product. It is natural to ask whether one can describe Gorenstein projective $\Lambda$-modules in terms of Gorenstein projective modules over $A$ and $B$. In this section, we shall give an approach to Gorenstein projective $\Lambda$-modules via monomorphism categories.

At fist, we use the monomorphism categories to describe the category of projective $\Lambda$-modules. We get the following result.

\begin{lem}\label{lemma-projective-bimodule-monocat}
Let $A$ and $B$ be two finite dimensional $k$-algebras. Assume that $k$ is a splitting field for $A$ or $B$. Then
$$\Mon(B, \projcat{A})=\projcat{\Lambda}=\Mon(A, \projcat{B}).$$
\end{lem}
\begin{proof}
We just prove $\projcat{\Lambda}=\Mon(B, \projcat{A})$. It is clear that $\projcat{\Lambda}\subseteq\Mon(B, \projcat{A})$. Thus we only need to show that any module $X$ in $\Mon(B, \projcat{A})$ is a projective $\Lambda$-module. By Lemma \ref{Lem-modules-over-tensor-product} $(3)$, we just need to show, for all simple left $A$-module $T$ and all simple left $B$-module $S$, $\Ext_{\Lambda}^i(X, T\otimes S)=0$  for all $i\ge1$. Since $X\in\Mon(B, \projcat{A})\subseteq\Mon(B, A)$, it follows Lemma \ref{lemma-monic-Ext-change-ring} that
$$\Ext_{\Lambda}^i(X, T\otimes S)\cong\Ext^i_{A}((A\otimes \D(S))\otimes_{\Lambda}X, T), $$
for all $i\ge 1$. By the definition of $\Mon(B, \projcat{A})$, we get that $(A\otimes \D(S))\otimes_{\Lambda}X$ is a projective $A$-module and then $\Ext^i_{A}((A\otimes \D(S))\otimes_{\Lambda}X, T)=0$, for all $i\ge 1$. Thus the proof is completed.
\end{proof}

Similarly, one can ask the following natural question:

\begin{ques}\label{Ques-Mon}
When does $\Lambda\Gp$ coincide with $\Mon(B, A\Gp)$?
\end{ques}

The following proposition shows that $\Mon(B, A\Gp)$ is always contained in $\Lambda\Gp$.

\begin{prop}\label{proposition-monic-in-gorenstein}
Let $A$ and $B$ be two finite dimensional $k$-algebras. Then $\Mon(B, A\Gp)\subseteq\Lambda\Gp$.
\end{prop}

 \begin{proof}
 Let $X$ be a $\Lambda$-module in $\Mon(B, A\Gp)$. Then $(A\otimes V)\otimes_{\Lambda}X\in A\Gp$ for all right $B$-modules $V$.  Thus, by Lemma \ref{lemma-monic-Ext-change-ring},  the Ext-group
   $$\Ext_{\Lambda}^i(X, A\otimes B)\cong \Ext_A^i((A\otimes \D(B))\otimes_{\Lambda}X, A)$$
vanishes for all $i\ge1$. Namely, $X\in {}^{\perp}\Lambda$.

Since $A\Gp$ has an injective cogenerator ${}_AA$, we deduce from Proposition \ref{proposition-inj-in-Mon(BL)} that $A\otimes B$ is an injective cogenerator in $\Mon(B, A\Gp)$.  Thus, for each $X$ in $\Mon(B, A\Gp)$, we can construct an exact sequence
$$0\lra X\lraf{f^1} M^1\lraf{f^2}M^2\lra\cdots$$
such that  $M^i\in\projcat{\Lambda}\subseteq {}^{\perp}\Lambda$ and $\mathrm{Cok} (f^i)\in\Mon(B, A\Gp)$ for all $i$.  Taking a projective resolution $\cdots \lra M^{-1}\lra M^0\lra X\lra 0$ of $X$, it is easy to see that $\cdots \lra M^{-1}\lra M^0\lra M^1\lra M^2\lra\cdots$ is a complete projective resolution of $X$. Hence $X\in\Lambda\Gp$.
 \end{proof}

 It remains to consider when is $\Lambda\Gp$ contained in $\Mon(B, A\Gp)$. We first obtain a necessary condition which refers to the property of CM-free. Recall that an algebra $\Gamma$ is called CM-free provided that $\Gamma\Gp=\projcat{\Gamma}$. Let us remark that a CM-free algebra does not necessarily have finite global dimension.

\begin{lem}\label{lemma-cm-free-monic=gorenstein}
If $\Lambda\Gp=\Mon(B, A\Gp)$, then $B$ is CM-free.
\end{lem}

\begin{proof}
Let $V\in B\Gp$. By Proposition~\ref{Prop-gorenstein-projective-bimodules}, $$A\otimes  V\in \Lambda\Gp=\Mon(B, A\Gp).$$ Since $\Mon(B, A\Gp)\subseteq \Mon(B, A)$, we have ${}_B(A\otimes  V)\in B\projcat$, thus $V\in B\projcat$.
Hence, $B\Gp=B\projcat$ and $B$ is CM-free.
\end{proof}

The following theorem deals with the case that $B$ has finite global dimension.

\begin{thm}\label{theorem-gldimB-finite-monic=gorenstein}
 Suppose that $B$ is Gorenstein. Then $\Lambda\Gp=\Mon(B, A\Gp)$ if and only if $\gldim(B)<\infty$.
 \end{thm}

 \begin{proof}
Suppose that  $X$ is a Gorenstein projective $\Lambda$-module.   Let $V$ be an arbitrary right $B$-module. We consider $N:=A\otimes V$. By adjointness, there  is an isomorphism $\Hom_A({}_AN, A)\cong A\otimes \D(V)$.

If $\gldim(B)<\infty$, it is easy to see that both $N_{\Lambda}$ and ${}_{\Lambda}\Hom_A({}_AN, A)$ have finite projective dimension. The module ${}_AN$ is clearly projective as a left $A$-module. By Lemma \ref{lemma-Gorenstein-Tor-Ext-property} (2) and (3), we obtain that $\Tor_i^{\Lambda}(A\otimes V, X)=0$ for all $i\ge1$, which implies that $X\in\Mon(B, A)$,  and that $(A\otimes V)\otimes_{\Lambda}X$ is a Gorenstein projective left $A$-module. Hence $X\in\Mon(B, A\Gp)$.  This proves that $\Lambda\Gp\subseteq \Mon(B, A\Gp)$. Together with Proposition \ref{proposition-monic-in-gorenstein}, we have $\Lambda\Gp=\Mon(B, A\Gp)$.

If $\Lambda\Gp=\Mon(B, A\Gp)$, then, by Lemma~\ref{lemma-cm-free-monic=gorenstein}, $B$ is CM-free. Combining that $B$ is Gorenstein, we get that $B$ has finite global dimension.
 \end{proof}

Next, we consider the case that $A$ is Gorenstein.

\begin{thm}\label{theorem-Gorenstein-CM-free}
Suppose that $A$ is Gorenstein. Then $\Lambda\Gp=\Mon(B, A\Gp)$ if and only if $B$ is CM-free.
\end{thm}

\begin{proof}
If $\Lambda\Gp=\Mon(B, A\Gp)$, then, by Lemma \ref{lemma-cm-free-monic=gorenstein}, $B$ is CM-free.

Now suppose that $B$ is CM-free and we will show that $\Lambda\Gp=\Mon(B, A\Gp)$.
Due to Proposition \ref{proposition-monic-in-gorenstein}, we only need to prove that $\Lambda\Gp\subseteq\Mon(B, A\Gp)$. Let ${}_{\Lambda}X$ be a Gorenstein projective $\Lambda$-module. By Proposition \ref{Prop-gorenstein-projective-bimodules}, both ${}_AX$ and ${}_BX$ are Gorenstein projective. Since $B$ is CM-free, the $B$-module ${}_BX$ is actually projective.  By Lemma \ref{lemma-monicRep-equiv-condition}, we deduce that $X\in\Mon(B, A)$.  It remains to prove that $(A\otimes V)\otimes_{\Lambda}X$ is Gorenstein projective as a left $A$-module for all right $B$-modules $V$.

Let $V$ be an arbitrary right $B$-module. Since $A$ is a Gorenstein algebra, the right $A$-module $\D({}_AA)$ has finite projective dimension, and thus $\D(A)\otimes B$ has finite projective dimension as a right $\Lambda$-module. By Lemma \ref{lemma-Gorenstein-Tor-Ext-property} (2), we have $\Tor_i^{\Lambda}(\D(A)\otimes B, X)=0$ for all $i\ge1$. Note that $\D(A)\otimes B$ is projective as a left $B$-module, that is, $\Ext_B^i(\D(A)\otimes B, \D(V))=0$ for all $i\ge1$. By Lemma \ref{lemma-basic-homological-isomorphism} (3), we have isomorphisms
$$\Ext_B^i((\D(A)\otimes B)\otimes_{\Lambda}X, \D(V))\cong\Ext_{\Lambda}^i(X, \Hom_B(\D(A)\otimes B, \D(V)))\cong \Ext_{\Lambda}^i(X, A\otimes \D(V))$$
for all $i\geq 0$. By Lemma \ref{lemma-monic-Ext-change-ring}, there is another natural isomorphism
$$\Ext_A^i((A\otimes V)\otimes_{\Lambda}X, A)\cong\Ext_{\Lambda}^i(X, A\otimes \D(V))$$
 for all $i\geq 0$. Hence there is an isomorphism
 $$\Ext_A^i((A\otimes V)\otimes_{\Lambda}X, A)\cong\Ext_B^i((\D(A)\otimes B)\otimes_{\Lambda}X, \D(V))$$
 for all $i\geq 0$.  The left $\Lambda$-module $\Hom_B(\D(A)\otimes B, B)\cong\Hom_k(\D(A), B)\cong A\otimes B\cong {}_{\Lambda}\Lambda$ is projective.  Thus the $B$-$\Lambda$-bimodule $\D(A)\otimes B$ satisfies the conditions in Lemma \ref{lemma-Gorenstein-Tor-Ext-property} (3). It follows that $(\D(A)\otimes B)\otimes_{\Lambda}X$ is Gorenstein projective as a left $B$-module. However, the algebra $B$ is CM-free. This forces  $(\D(A)\otimes B)\otimes_{\Lambda}X$ to be projective as a left $B$-module. Hence $\Ext_B^i((\D(A)\otimes B)\otimes_{\Lambda}X, \D(V))=0$ for all $i\ge1$. Consequently, $\Ext_A^i((A\otimes V)\otimes_{\Lambda}X, A)=0$ for all $i\ge1$. That is $(A\otimes V)\otimes_{\Lambda}X\in {}^{\perp}A$. Since $A$ is a Gorenstein algebra, we deduce that $(A\otimes V)\otimes_{\Lambda}X$ is Gorenstein projective as an left $A$-module.

Altogether, we have proved that $X\in\Mon(B, A)$ and $(A\otimes V)\otimes_{\Lambda}X$ is Gorenstein projective as a left $A$-module for all right $B$-modules $V$. Hence $X\in\Mon(B, A\Gp)$. This finishes the proof.
\end{proof}

\begin{rem} $(1).$ Theorem~\ref{theorem-gldimB-finite-monic=gorenstein} and~\ref{theorem-Gorenstein-CM-free} show that the answer to Question~\ref{Ques-Mon}  is negative in general.

$(2).$ Note that $\projcat{\Lambda}\subseteq\Mon(B, A\Gp)\subseteq\Lambda\Gp$ and both $\Mon(B, A\Gp)$ and $\Lambda\Gp$ are Frobenius exact categories, we get that the stable category $\stMon(B, A\Gp)$ is a thick triangulated subcategory of the stable category $\Lambda\stGp$. Thus, there is a natural question: How to describe the Verdier quotient $\Lambda\stGp/\stMon(B, A\Gp)$?
Theorem~\ref{theorem-gldimB-finite-monic=gorenstein} and~\ref{theorem-Gorenstein-CM-free}  give  some partial answers to this question.
\end{rem}

Now, we get the following sequence of full subcategories of $\modcat{\Lambda}$,
$$\projcat{\Lambda}\subseteq\Mon(B, \projcat{A})\subseteq\Mon(B, A\Gp)\subseteq\Lambda\Gp.$$
It is clear that if $\Lambda$ is CM-free, namely, $\Lambda\Gp=\projcat{\Lambda}$, then $\Mon(B, \projcat{A})=\Mon(B, A\Gp)$ and then $A\Gp=\projcat{A}$. Thus we get the following corollary,

\begin{cor}\label{corollary-cmfree-tensor-product}
Let $A$ and $B$ be two finite dimensional $k$-algebras. If $\Lambda=A\otimes B$ is CM-free, then both $A$ and $B$ are CM-free.
\end{cor}

Conversely, we have the following partial answer.

\begin{cor}\label{corollary-cmfree-tensor-product-equivalent-condition}
Let $A$ and $B$ be two finite dimensional $k$-algebras. Assume that $k$ is a splitting field for $A$ or $B$. If $A$ or $B$ is Gorenstein, then $\Lambda=A\otimes B$ is CM-free if and only if both $A$ and $B$ are CM-free.
\end{cor}

\begin{proof}
At first, by Lemma \ref{lemma-projective-bimodule-monocat}, $\projcat{\Lambda}=\Mon(B, \projcat{A})$.
Assume that both $A$ and $B$ are CM-free. If $A$ is Gorenstein, then by Theorem \ref{theorem-Gorenstein-CM-free}, we get $\Lambda\Gp=\Mon(B, A\Gp)=\Mon(B, \projcat{A})=\projcat{\Lambda}$. Thus $\Lambda$ is CM-free. If $B$ is Gorenstein, then $B$ has finite global dimension, and then by Theorem \ref{theorem-gldimB-finite-monic=gorenstein}, we get $\Lambda\Gp=\Mon(B, A\Gp)=\Mon(B, \projcat{A})=\projcat{\Lambda}$. Thus $\Lambda$ is CM-free.
\end{proof}

We propose the following conjecture:

\medskip
{\bf\noindent Conjecture}: {\it Let $A$ and $B$ be two finite dimensional $k$-algebras. Then $\Lambda=A\otimes B$ is CM-free if and only if both $A$ and $B$ are CM-free.}

\medskip

Theorem \ref{theorem-gldimB-finite-monic=gorenstein} and \ref{theorem-Gorenstein-CM-free} also give a convenient way to describe Gorenstein projective bimodules. In the following example, one can write down all the indecomposable Gorenstein projective modules explicitly via  monic representations.

\begin{exm}\label{example-monomorphism-category}
Let $A$ be $k[x]/(x^2)$, and let $B$ be the algebra given by the quiver $\xymatrix@ru@R=4mm@C=4mm{{\bullet}\ar[r]^{\alpha}\ar[d]_{\beta} &  {\bullet}\ar[d]^{\gamma} \\  {\bullet}\ar[r]_{\delta} & {\bullet} }$ with relation $\gamma\alpha-\delta\beta$. The algebra $B$ is a tilted algebra and has finite global dimension. By Theorem \ref{theorem-gldimB-finite-monic=gorenstein}, the Gorenstein projective $(A\otimes B)$-modules are precisely those modules in $\Mon(B, A\Gp)$. Note that, by a result of Rickard \cite{Rickard1991},  $A\otimes B$ is derived equivalent to $A\otimes kQ$, where $Q$ is a Dynkin quiver of type $D_4$. Ringel and Zhang \cite{Ringel2011-preprint} have proved that the number of indecomposable non-projective Gorenstein projective modules over $A\otimes kQ$ is the same as the number of indecomposable $kQ$-modules, which is $12$ when $Q$ is of type $D_4$. It is also well-known that derived equivalences preserve stable categories of Gorenstein projective modules (see, for example, \cite{Hu2016}). Thus the algebra $A\otimes B$ also has 12 indecomposable non-projective Gorenstein projective modules. By using  monic representations, it is very easy to write down all these modules. Actually, a module $X$ over $A\otimes B$ is a monic representation of $B$ over $A$ if and only if both $X_{\alpha}: X_1\lra X_2$ and $X_{\beta}: X_1\lra X_3$ are monomorphisms and the sequence of $A$-modules
$$\xymatrix@C=12mm{0\ar[r] & X_1\ar[r]^-{[X_{\alpha},\; X_{\beta}]^{T}} & X_2\oplus X_3\ar[r]^-{[-X_{\gamma},\; X_{\delta}]} &X_4}$$ is exact.
The simple $A$-module is denoted by $S$, and $\lambda: S\lra A$ and $\pi: A\lra S$ are the canonical inclusion and surjective $A$-maps respectively. The indecomposable non-projective Gorenstein projective $(A\otimes B)$-modules are as follows.
 $$
 \xy
(0,0)*+{0}="1",
(15,0)*+{0}="2",
(0,-12)*+{0}="3",
(15,-12)*+{S}="4",
{\ar "1";"2"},
{\ar "1";"3"},
{\ar "3";"4"},
{\ar   "2";"4"},
\endxy \quad \quad
 \xy
(0,0)*+{0}="1",
(15,0)*+{S}="2",
(0,-12)*+{0}="3",
(15,-12)*+{S}="4",
{\ar "1";"2"},
{\ar "1";"3"},
{\ar "3";"4"},
{\ar^{1} "2";"4"},
\endxy \quad \quad
 \xy
(0,0)*+{0}="1",
(15,0)*+{0}="2",
(0,-12)*+{S}="3",
(15,-12)*+{S}="4",
{\ar "1";"2"},
{\ar "1";"3"},
{\ar "2";"4"},
{\ar^{1} "3";"4"},
\endxy \quad \quad
 \xy
(0,0)*+{0}="1",
(15,0)*+{S}="2",
(0,-12)*+{0}="3",
(15,-12)*+{A}="4",
{\ar "1";"2"},
{\ar "1";"3"},
{\ar "3";"4"},
{\ar^{\lambda} "2";"4"},
\endxy \quad \quad
 \xy
(0,0)*+{0}="1",
(15,0)*+{0}="2",
(0,-12)*+{S}="3",
(15,-12)*+{A}="4",
{\ar "1";"2"},
{\ar "1";"3"},
{\ar "2";"4"},
{\ar^{\lambda} "3";"4"},
\endxy \quad \quad
\xy
(0,0)*+{0}="1",
(15,0)*+{S}="2",
(0,-12)*+{S}="3",
(15,-12)*+{A\oplus S}="4",
{\ar "1";"2"},
{\ar "1";"3"},
{\ar^{[\lambda, 1]^T} "2";"4"},
{\ar^(.45){[\lambda, 0]^T} "3";"4"},
\endxy $$
$$\xy
(0,0)*+{S}="1",
(15,0)*+{S}="2",
(0,-12)*+{S}="3",
(15,-12)*+{S}="4",
{\ar^{1} "1";"2"},
{\ar^{1} "1";"3"},
{\ar^{1} "2";"4"},
{\ar^{1} "3";"4"},
\endxy \quad  \xy
(0,0)*+{S}="1",
(0,-12)*+{S}="3",
(15,0)*+{A}="2",
(15,-12)*+{A}="4",
{\ar^{1} "1";"3"},
{\ar^{\lambda} "1";"2"},
{\ar^{\lambda} "3";"4"},
{\ar^{1} "2";"4"},
\endxy \quad
 \xy
(0,0)*+{S}="1",
(15,0)*+{S}="2",
(0,-12)*+{S}="3",
(15,-12)*+{A}="4",
{\ar^{1} "1";"2"},
{\ar_{1} "1";"3"},
{\ar^{\lambda} "2";"4"},
{\ar^{\lambda} "3";"4"},
\endxy \quad \xy
(0,0)*+{S}="1",
(15,0)*+{S}="2",
(0,-12)*+{A}="3",
(15,-12)*+{A}="4",
{\ar^{1} "1";"2"},
{\ar_{\lambda} "1";"3"},
{\ar^{\lambda} "2";"4"},
{\ar^{1} "3";"4"},
\endxy \quad
 \xy
(0,0)*+{S}="1",
(15,0)*+{A}="2",
(0,-12)*+{A}="3",
(15,-12)*+{A\oplus S}="4",
{\ar^{\lambda} "1";"2"},
{\ar_{\lambda} "1";"3"},
{\ar^{[1, \pi]^{T}} "2";"4"},
{\ar^(.45){[1, 0]^{T}} "3";"4"},
\endxy \quad
 \xy
(0,0)*+{S}="1",
(15,0)*+{A}="2",
(0,-12)*+{A}="3",
(15,-12)*+{A\oplus A}="4",
{\ar^{\lambda} "1";"2"},
{\ar_{\lambda} "1";"3"},
{\ar^{[1, \lambda\pi]^{T}} "2";"4"},
{\ar^(.45){[1, 0]^{T}} "3";"4"},
\endxy$$
\end{exm}

At the end of this section, we give an application of Proposition~\ref{proposition-monic-in-gorenstein}.

\begin{exm}\label{example-application-to-morita-context}
  Let $A$ be a finite dimensional algebra, and let $\Lambda:=\left[\begin{smallmatrix} A & A \\ A & A\end{smallmatrix}\right]$ be the Morita context whose multiplication is given by
  $$\begin{bmatrix} a & b \\ c & d\end{bmatrix}\cdot \begin{bmatrix} a' & b' \\ c' & d'\end{bmatrix}=\begin{bmatrix} aa' & ab'+bd' \\ ca'+dc' & dd'\end{bmatrix}.$$
Actually, the algebra $\Lambda$ can be viewed as a tensor product of $A$ and the algebra $B$ given by the quiver  $Q: \xymatrix{1\ar@<1mm>[r]^{\alpha} & 2\ar@<1mm>[l]^{\beta}}$ with relations $\alpha\beta=0=\beta\alpha$. Thus every $\Lambda$-module $X$ can be viewed as a representation of $Q$ over $A\modcat$, namely, $\xymatrix{X_1\ar@<1mm>[r]^{X_{\alpha}} & X_2\ar@<1mm>[l]^{X_{\beta}}}$ with $X_{\beta}X_{\alpha}=0=X_{\alpha}X_{\beta}$, where $X_1$ and $X_2$ are $A$-modules, $X_{\alpha}$ and $X_{\beta}$ are $A$-maps. By the definition of  monic representations of $B$ over $A$, it is easy to see that $X$ is in $\Mon(B, A)$ if and only if the sequence of $A$-modules $\xymatrix@1{X_1\ar[r]^{X_{\alpha}} & X_2\ar[r]^{X_{\beta}} & X_1\ar[r]^{X_{\alpha}} & X_2}$  is exact. The module $X$ belongs to $\Mon(B, A\Gp)$ if and only if $X$ further satisfies the condition that both $X_{\alpha}$ and $X_{\beta}$ have kernels in $A\Gp$. Thus Proposition \ref{proposition-monic-in-gorenstein}  gives another proof of the result \cite[Corollary 3.12]{GaoPsaroudakis2015}.
\end{exm}

\section{Gorenstein projective bimodules via filtration categories}

\subsection{Finitely generated Gorenstein projective modules}

To study the Gorenstein projective modules over tensor products, another strategy  is to describe  $(A\otimes B)\Gp$ in terms of $A\Gp$ and $B\Gp$.  As before, we write $\Lambda$ for $A\otimes B$ throughout this subsection. At first,  for arbitrary $X\in A\Gp$ and $Y\in B\Gp$, we have $X\otimes Y\in \Lambda\Gp$ by Proposition \ref{Prop-gorenstein-projective-bimodules}. However, in general, Gorenstein projective $\Lambda$-modules may not be of this form.  For instance, if both $A$ and $B$ are selfinjective, then so is $\Lambda$. In this case $\Lambda\Gp=\modcat{\Lambda}$. In general, there are $\Lambda$-modules which are not tensor products of $A$-modules and $B$-modules.

\medskip
 Let $\Gamma$ be an algebra, and let $\mathscr{X}$ be a class of $\Gamma$-modules.  We denote by $\filt(\mathscr{X})$ the full subcategory of $\modcat{\Gamma}$ consisting of module $X$ admitting a filtration  $0=X_0\subset X_1\subset\cdots\subset X_m=X$ of $\Gamma$-modules such that the factors $X_i/X_{i-1}$ are all in $\mathscr{X}$ for all $1\leq i\leq m$. By $\addfilt(\mathscr{X})$ we denote the additive closure of $\filt(\mathscr{X})$. Precisely speaking, $\addfilt(\mathscr{X})$ consists of modules which are direct summands of modules in $\filt(\mathscr{X})$. Actually, the category $\filt(\mathscr{X})$ is the smallest full subcategory of $\modcat{\Gamma}$ containing $\mathscr{X}$ closed under extensions, and $\addfilt(\mathscr{X})$ is the smallest full subcategory of $\modcat{\Gamma}$ containing $\mathscr{X}$ closed under extensions and direct summands.

\medskip
For simplicity, given $\mathscr{X}\subseteq \modcat{A}$ and $\mathscr{Y}\subseteq \modcat{B}$, we write
 $$\mathscr{X}\otimes\mathscr{Y}:=\{X\otimes Y\in\modcat{\Gamma}\ |  \ X\in\mathscr{X}, Y\in\mathscr{Y}\}.$$
We already know that $A\Gp\otimes B\Gp$ is contained in $\Lambda\Gp$.  Since $\Lambda\Gp$ is closed under extensions and direct summands, it must contain $\addfilt(A\Gp\otimes B\Gp)$ as a full subcategory.  The naive question here is:

\smallskip
\begin{ques}\label{Ques-filt}
Does $\Lambda\Gp$ coincide with $\addfilt(A\Gp\otimes B\Gp)$?
\end{ques}

\smallskip
\noindent
If the above question has a positive answer, then we get a satisfactory description of Gorenstein projective modules over tensor product algebras. Our first answer to this question is the following result.

\begin{thm}\label{theorem-Gorenstein-tensor-Gorenstein}
Let  $A$ and $B$ be  Gorenstein algebras. Assume that $k$ is a splitting field for   $A$ or $B$.  Then
$$(A\otimes B)\Gp=\addfilt(A\Gp\otimes B\Gp).$$
\end{thm}

Before giving the proof, we fix some notation. For each algebra $\Gamma$ and $m\geq 0$, we write
$$\Omega_{\Gamma}^m:=\addfilt\big(\{\Omega_{\Gamma}^m(X)|X\in\modcat{\Gamma}\}\cup  \projcat{\Gamma}\big),$$
where for any $X\in\modcat{\Gamma}$, $\Omega^0_{\Gamma}(X)=X$, $\Omega_{\Gamma}(X)$ is the kernel of the projective cover $f: {}_{\Gamma}P\to {}_{\Gamma}X$, and $\Omega^{i+1}_{\Gamma}(X)=\Omega_{\Gamma}(\Omega^i_{\Gamma}(X))$ for all $i\ge1$. Clearly $\Omega_{\Gamma}^0=\modcat{\Gamma}$ and $\Omega_{\Gamma}^{m+1}\subseteq\Omega_{\Gamma}^m$. Moreover,  it is easy to see that, if $\Gamma$ is Gorenstein with $\id {}_{\Gamma}\Gamma=d<\infty$, then $\Omega_{\Gamma}^d=\Gamma\Gp$.  The following lemma is useful for the proof  of Theorem \ref{theorem-Gorenstein-tensor-Gorenstein}.

\begin{lem}\label{lemma-OmegaXY-filt}
Let $A$ and $B$ be two algebras, and let $d_A, d_B$ be non-negative integers.   Suppose that $X\in\modcat{A}$ and $Y\in\modcat{B}$. Then $\Omega_{A\otimes B}^{d_A+d_B}(X\otimes Y)\in\addfilt(\Omega_A^{d_A}\otimes\Omega_B^{d_B})$.
\end{lem}

\begin{proof}
We use induction of the sum $d_A+d_B$.  The case that $d_A+d_B=0$ is clear.  Now assume that $d_A+d_B>0$. Without loss of generality, we assume that $d_A>0$. Then by induction hypothesis, we have $\Omega_{A\otimes B}^{d_A-1+d_B}(X\otimes Y)\in\addfilt(\Omega_A^{d_A-1}\otimes\Omega_B^{d_B})$. Now for arbitrary $U\in\Omega_A^{d_A-1}$ and $V\in\Omega_B^{d_B}$. Let $0\ra \Omega_A(U)\ra P\ra U\ra 0$ and $0\ra \Omega_B(V)\ra Q\ra V\ra 0$ be short exact sequences such that $P$ is a projective cover of ${}_AU$ and $Q$ is a projective cover of ${}_BV$.  Then we can form the following commutative diagram with exact rows and columns.
$$\xymatrix{
& 0\ar[d] & 0\ar[d]\\
& P\otimes \Omega_B(V)\ar@{=}[r]\ar[d] & P\otimes \Omega_B(V)\ar[d]\\
0 \ar[r] & L\ar[r]\ar[d] & P\otimes Q\ar[r]\ar[d] & U\otimes V\ar@{=}[d]\ar[r] & 0\\
0\ar[r] & \Omega_A(U)\otimes V\ar[r]\ar[d] & P\otimes V\ar[r]\ar[d] & U\otimes V\ar[r] & 0\\
& 0& 0
}$$
Since $U\in\Omega_A^{d_A-1}$, by the Horseshoe lemma, we have $\Omega_A(U)\in\Omega_A^{d_A}$. Recall that $V\in\Omega_B^{d_B}$. Then $\Omega_B(V)\in\Omega_B^{d_B+1}\subseteq\Omega_B^{d_B}$.
 It follows that both $\Omega_A(U)\otimes V$ and $P\otimes \Omega_B(V)$ lie in $\Omega_A^{d_A}\otimes \Omega_B^{d_B}$.  Hence $\Omega_{A\otimes B}(U\otimes V)$, which is a direct summand of $L$, belongs to $\Omega_A^{d_A}\otimes \Omega_B^{d_B}$. Hence $\Omega_{A\otimes B}\big(\addfilt(\Omega_A^{d_A-1}\otimes\Omega_B^{d_B})\big)\subseteq\addfilt(\Omega_A^{d_A}\otimes \Omega_B^{d_B})$ by the Horseshoe lemma.  Thus we conclude that $$\Omega_{A\otimes B}^{d_A+d_B}(X\otimes Y)=\Omega_{A\otimes B}(\Omega_{A\otimes B}^{d_A-1+d_B}(X\otimes Y))\in\addfilt(\Omega_A^{d_A}\otimes \Omega_B^{d_B}).$$
This finishes the proof.
\end{proof}

Now we can give a proof of Theorem \ref{theorem-Gorenstein-tensor-Gorenstein}.

\begin{proof}[{\bf Proof of Theorem \ref{theorem-Gorenstein-tensor-Gorenstein}}]
Suppose that $\id({}_AA)=d_A$ and $\id({}_BB)=d_B$. Let $\Lambda:=A\otimes B$. Then $\Lambda$ is a Gorenstein algebra and  $\id({}_{\Lambda}\Lambda)=d_A+d_B$. Since $k$ is a splitting field for   $A$ or $B$, by Lemma~\ref{Lem-modules-over-tensor-product} (3) every simple $\Lambda$-module is of the form $T\otimes S$, where ${}_AT$ is a simple $A$-module and ${}_BS$ is a simple $B$-module. Hence $\modcat{\Lambda}=\addfilt(\modcat{A}\otimes \modcat{B})$.  By Lemma \ref{lemma-OmegaXY-filt}, we have
$$\Omega_{\Lambda}^{d_A+d_B}\subseteq\addfilt(\Omega_A^{d_A}\otimes\Omega_B^{d_B}).$$
Note that $A\Gp=\Omega_A^{d_A}$, $B\Gp=\Omega_B^{d_B}$ and $\Lambda\Gp=\Omega_{\Lambda}^{d_A+d_B}$. Hence $\Lambda\Gp\subseteq\addfilt(A\Gp\otimes B\Gp)$.  The inclusion in the other direction follows from Proposition \ref{Prop-gorenstein-projective-bimodules} (1).
\end{proof}

The proof of Theorem \ref{theorem-Gorenstein-tensor-Gorenstein} certainly does not work if either $A$ or $B$ is not Gorenstein.  In the following, we assume that $A$ is an arbitrary algebra, and consider under which conditions on $B$ can we get an affirmative answer to the above question. The main tool is our theory on monic representations developed above.  The following lemma is crucial.

\begin{lem}\label{lemma-triangular-matrix-algebra-reduction}
Let $A$ be an algebra, and let $B:=\left[\begin{smallmatrix}B_1& M\\0 & B_2\end{smallmatrix}\right]$ be an upper triangular matrix algebra.  Suppose that $$\Mon(B_i, A\Gp)=\filt(A\Gp\otimes \projcat{B_i})$$ for $i=1, 2$. Then $\Mon(B, A\Gp)=\filt(A\Gp\otimes \projcat{B})$.
\end{lem}

\begin{proof}
Let $e$ be the idempotent $\left[\begin{smallmatrix} 0 & 0 \\0 & 1 \end{smallmatrix}\right]$ in $B$. It is easy to see that $eB=eBe$ and  thus $BeB=BeBe=Be$. It follows that the multiplication map $Be\otimes_{eBe}eB\ra BeB$ is an isomorphism, and we get an exact sequence of $B$-$B$-bimodules
$$0\lra Be\otimes_{eBe}eB\lra B\lra B/BeB\lra 0.$$
As above, we set $\Lambda:=A\otimes B$. Let $X$ be a $\Lambda$-module in $\Mon(B, A\Gp)$. Applying $(A\otimes -)\otimes_{\Lambda}X$ to the above exact sequence gives rise to the following exact sequence of $\Lambda$-modules.
$$0\lra (A\otimes (Be\otimes_{eBe}eB))\otimes_{\Lambda}X\lra (A\otimes B)\otimes_{\Lambda}X\lra (A\otimes (B/BeB))\otimes_{\Lambda}X\lra 0. \quad\quad (*)$$
The middle term is isomorphic to $X$ itself.  We shall prove that the terms on both sides belong to $\filt(A\Gp\otimes \projcat{B})$.

Note that $eB(=eBe)$ is projective as left $eBe$-module. We have $(A\otimes eB)\otimes_{\Lambda}X\in\Mon(eBe, A\Gp)$ by Proposition \ref{proposition-Mon(BL)-to-Mon(CL)}. Since $eBe\cong B_2$, by our assumption, we get $(A\otimes eB)\otimes_{\Lambda}X\in\filt(A\Gp\otimes\projcat{eBe})$. Clearly, the functor $(A\otimes Be)\otimes_{(A\otimes eBe)}-$ sends modules in $\filt(A\Gp\otimes\projcat{eBe})$ to modules in $\filt(A\Gp\otimes \projcat{B})$.  Thus $(A\otimes (Be\otimes_{eBe}eB))\otimes_{\Lambda}X$, which is isomorphic to $(A\otimes Be)\otimes_{(A\otimes eBe)}(A\otimes eB)\otimes_{\Lambda}X$, falls into $\filt(A\Gp\otimes\projcat{B})$.  This proves that the term on the left hand side of $(*)$ is in $\filt(A\Gp\otimes\projcat{B})$.

For simplicity, we write $\bar{B}$ for $B/BeB$. By Proposition \ref{proposition-Mon(BL)-to-Mon(CL)} again, the $(A\otimes \bar{B})\otimes_{\Lambda}X$ belongs to $\Mon(\bar{B}, A\Gp)$. Note that the canonical surjective algebra homomorphism $B\lra \bar{B}$ induces an isomorphism $B_1\cong\bar{B}$, and $B_1$ is projective as a left $B$-module. This implies, by our assumption, that  $(A\otimes \bar{B})\otimes_{\Lambda}X\in\filt(A\Gp\otimes\projcat{\bar{B}})$. And every projective left $\bar{B}$-module is projective as a left $B$-module via the canonical map $B\lra \bar{B}$. Consequently, every module in $\filt(A\Gp\otimes\projcat{\bar{B}})$ is in $\filt(A\Gp\otimes\projcat{B})$ as a left $\Lambda$-module. This proves that the term on the right hand side of $(*)$ is in $\filt(A\Gp\otimes\projcat{B})$.

Altogether, we have proved that the terms on both sides of $(*)$ are in $\filt(A\Gp\otimes\projcat{B})$. Hence $X$, which is isomorphic to the middle term of $(*)$,  falls into $\filt(A\Gp\otimes\projcat{B})$. This finishes the proof.
\end{proof}

Recall that an algebra $B$ is called a triangular algebra provided that  the Ext-quiver of $B$ has no oriented cycles.

\begin{thm}\label{theorem-B-triangular}
Let $A$ be an algebra, and let $B$ be a triangular algebra such that $k$ is a splitting field for $B$. Then
$$\Mon(B, A\Gp)=(A\otimes B)\Gp=\filt(A\Gp\otimes B\Gp).$$
\end{thm}

\begin{proof}
It is well-known that a triangular algebra has finite global dimension, then $B\Gp=\projcat{B}$. By Theorem \ref{theorem-gldimB-finite-monic=gorenstein}, we have $(A\otimes B)\Gp=\Mon(B, A\Gp)$. It remains to prove that $\Mon(B, A\Gp)=\filt(A\Gp\otimes\projcat{B})$.  By Corollary \ref{corollary-Morita-BL-CL}, we can assume that $B$ is a basic algebra. Suppose that $B$ has $n$ pairwise non-isomorphic simple $B$-modules.  We use induction on $n$. If $n=1$, then $B=k$. In this case, it is obvious that $\Mon(B, A\Gp)=\filt(A\Gp\otimes\projcat{B})$. Now assume that $n>1$. Since $B$ is triangular algebra, $B$ is isomorphic to a triangular matrix algebra $\left[\begin{smallmatrix}B_1& M\\0 & B_2\end{smallmatrix}\right]$ such that both  $B_1$ and $B_2$ are triangular algebras with less simple modules. By induction hypothesis, we can assume that $\Mon(B_i, A\Gp)=\filt(A\Gp\otimes\projcat{B_i})$ for $i=1, 2$. It follows from Lemma \ref{lemma-triangular-matrix-algebra-reduction} that $\Mon(B, A\Gp)=\filt(A\Gp\otimes\projcat{B})=\filt(A\Gp\otimes B\Gp)$.
\end{proof}

Despite Theorem~\ref{theorem-Gorenstein-tensor-Gorenstein} and Theorem~\ref{theorem-B-triangular},  the answer to  Question~\ref{Ques-filt} is negative in general. We are grateful to Xiao-Wu Chen and Dawei Shen for communicating us the following example.

\begin{exm}\label{example-filtration-fail}
Suppose $\mbox{\rm char}\ k=0$. Let $A=k[X]/(X^2)$ and $B=k[X_1, \cdots, X_n]/(X_1, \cdots, X_n)^2$ with $n\geq 2$.
Since $A$ is selfinjective, $\modcat{A}= A\Gp$ and there are only two indecomposable $A$-modules, namely $k$ and $A$.
The algebra  $B$ is CM-free by \cite{Chen2012f}, so $B\Gp=B\projcat$ and  the regular module ${}_BB$ is the  unique indecomposable Gorenstein projective left $B$-module. It is easy to see that, up to isomorphism, there are finitely many modules  in $\addfilt(A\Gp\otimes B\Gp)$ of a fixed dimension.  However,  the results in {\rm \cite{Tracy2015-reprint}} show that there are infinitely many isomorphism classes of Gorenstein projective $(A\otimes B)$-modules of certain given dimension. Hence $\addfilt(A\Gp\otimes B\Gp)\subsetneq(A\otimes B)\Gp$. But, by Theorem \ref{theorem-Gorenstein-CM-free}, we get $(A\otimes B)\Gp=\Mon(B, A\Gp)$.
\end{exm}

\subsection{Infinitely generated Gorenstein projective modules}

In this subsection, we go beyond the other part of this paper by considering infinitely generated Gorenstein projective modules. We want to show a version of
Theorem~\ref{theorem-Gorenstein-tensor-Gorenstein}  for infinitely generated Gorenstein projective modules. At first, let us recall the relevant definitions.

Let $A$ be a k-algebra, not necessarily finite-dimensional. Denote by $A\Modcat$, $A\Projcat$ the category of all left $A$-modules and that of all projective left $A$-modules, respectively.
An $A$-module $M\in \Modcat{A}$ is Gorenstein projective if
 there is an exact complex
$$\cpx{P}: \quad\cdots\lra P^{i}\lraf{d^i} P^{i+1}\lraf{d^{i+1}} P^{i+2}\lra\cdots$$
 of (not necessarily finitely generated) projective  modules $P^i\in \Projcat{A}$ with $X=\Ker(d^1)$ such that $\Hom_A(\cpx{P}, Q)$ is again exact for arbitrary projective $A$-module $Q$. Let $A\GP$ denote the full subcategory of $A\Modcat$ consisting of all Gorenstein projective $A$-modules.

\medskip

Let $\mathscr{X}$ be a class of $A$-modules.  We denote by $\Filt(\mathscr{X})$ the full subcategory of $A\Modcat$ consisting of module $X$ such that there exists an ordinal $\alpha$ and  a filtration  $0=X_0\subset X_1\subset\cdots\subset X_\beta\subset X_{\beta+1}\subset \cdots \subset X_\alpha=X$ of $A$-modules  with  the factors $X_{\beta+1}/X_{\beta}$  all lying  in $\mathscr{X}$ for all $\beta+1\leq \alpha$. By $\AddFilt(\mathscr{X})$ we denote the additive closure of $\Filt(\mathscr{X})$. Precisely speaking, $\AddFilt(\mathscr{X})$ consists of modules which are direct summands of modules in $\Filt(\mathscr{X})$. Actually, the category $\Filt(\mathscr{X})$ is the smallest full subcategory of $\Modcat{A}$ containing $\mathscr{X}$ closed under (infinite) extensions, and $\AddFilt(\mathscr{X})$ is the smallest full subcategory of $A\Modcat$ containing $\mathscr{X}$ closed under (infinite) extensions and direct summands.

\begin{thm}\label{theorem-Gorenstein-tensor-Gorenstein-infinite-version}
Let  $A$ and $B$ be finite dimensional  Gorenstein $k$-algebras. Assume that $k$ is a splitting field for   $A$ or $B$.  Then $(A\otimes B)\GP=\AddFilt(A\GP\otimes B\GP)$.
\end{thm}

The proof of the above result uses essentially Quillen's small object argument. Let us recall one version of it. Let $R$ be a ring.  Let $\mathscr{C}$ be a class of modules in $R\Modcat$. Denote
$$\mathscr{C}^{\perp_1}=\{M\in \Modcat{R}\  |\ \mathrm{Ext}^1_R(  \mathscr{C}, M)=0\}$$ and
$$ {}^{\perp_1}\mathscr{C}=\{M\in \Modcat{R}\  |\ \mathrm{Ext}^1_R(M,   \mathscr{C})=0\}.$$

Let $\mathscr{S}$ be a set of modules in $\Modcat{R}$. Quillen's small object argument shows that  it cogenerates a complete cotorsion pair $({}^{\perp_1}(\mathscr{S}^{\perp_1}), \mathscr{S}^{\perp_1})$ and
${}^{\perp_1}(\mathscr{S}^{\perp_1})=\AddFilt(\mathscr{S})$; see \cite[Theorem 6.5]{Hovey2002}.  Another proof was given by Eklof and Trlifaj \cite[Theorem 10]{Eklof2001}.


\begin{proof}

By \cite[Theorem 8.3]{Hovey2002}, for a Gorenstein algebra $R$ of with $\id{}_RR=d$, $$R\GP= {}^{\perp_1}(\{\Omega^{d}(R/I)\ |\ I \ \mathrm{is \ a\  left \ ideal\  of}\ R\}^{\perp_1}).$$
Now, given two finite-dimensional Gorenstein algebras $A$ and $B$ with  $\id{}_AA=d_A$ and $\id{}_BB=d_B$ respectively, the tensor product algebra $A\otimes B$ is a Gorenstein algebra with $\id{}_{A\otimes B}(A\otimes B)=d_A+d_B$ and $$(A\otimes B)\GP={}^{\perp_1}(\{\Omega^{d_A+d_B}((A\otimes B)/I)\ |\ I \ \mathrm{is \ a\  left \ ideal\  of}\ A\otimes B\}^{\perp_1}).$$
Since $A\otimes B$ is a finite dimensional algebra, $(A\otimes B)/I$ is a finite extension of simple $(A\otimes B)$-modules,  and by assumption,  simple $(A\otimes B)$-modules are of the form $T\otimes S$ for a simple $A$-module $T$ and a simple $B$-module $S$. By a similar proof as that of Lemma~\ref{lemma-OmegaXY-filt},
we see that $$(A\otimes B)\GP={}^{\perp_1}(\{\Omega^{d_A+d_B}((A\otimes B)/I)\ |\ I \ \mathrm{is \ a\  left \ ideal\  of}\ A\otimes B\}^{\perp_1})= {}^{\perp_1}((\Omega_A^{d_A}\otimes\Omega_B^{d_B})^{\perp_1}).$$
By Quillen's small object argument, we have
$$ (A\otimes B)\GP=\AddFilt(\Omega_A^{d_A}\otimes\Omega_B^{d_B})\subseteq \AddFilt(A\GP\otimes B\GP).$$
The inclusion $\AddFilt(A\GP\otimes B\GP)\subseteq (A\otimes B)\GP$ follows from \cite[Lemma 1]{Eklof2001}, because   $$(A\otimes B)\GP={}^{\perp_1}(\{\Omega^{d}((A\otimes B)/I)\ |\ I \ \mathrm{is \ a\  left \ ideal\  of}\ A\otimes B\}^{\perp_1})$$ is closed under transfinite extensions and direct summands.
So $(A\otimes B)\GP=\AddFilt(A\GP\otimes B\GP)$.
\end{proof}

\section*{Acknowledgements}

The  author W. Hu is grateful to NSFC (No.11471038, No.11331006) and the Fundamental Research Funds for the Central Universities for partial support. X.-H. Luo is supported by NSFC (No.11401323, No. 11771272) and Jiangsu Government Scholarship for Oversea Studies. B.-L. Xiong is supported by NSFC (No.11301019, No.11471038). G. Zhou is supported by NSFC (No.11671139) and by STCSM (No.13dz2260400). The main part of this work was completed during the author B.-L. Xiong visiting the University of Bielefeld supported by China Scholarship Council (CSC No.201506885017). He wishes to express his gratitude to Henning Krause for the cordial hospitality and the wonderful working atmosphere.

\bigskip

{\footnotesize
Wei Hu

\medskip
School of Mathematical Sciences, Laboratory of Mathematics and Complex Systems, MOE, Beijing Normal University, 100875 Beijing, China

{\tt Email: huwei@bnu.edu.cn}

\bigskip
Xiu-Hua Luo

\medskip
Department of Mathematics, Nantong University,
Jiangsu 226019, P. R. China

{\tt Email: xiuhualuo2014@163.com }

\bigskip
Bao-Lin Xiong

\medskip

Beijing No. 4 High School, Beijing 100034, P. R. China

\smallskip
Department of Mathematics, \ Beijing  University of Chemical  Technology,
Beijing 100029, P. R. China

{\tt Email: xiongbaolin@gmail.com}

\bigskip
Guodong Zhou

\medskip
School of Mathematical Sciences, Shanghai Key Laboratory of PMMP, East China Normal University, Shanghai 200241, P. R. China

{\tt Email: gdzhou@math.ecnu.edu.cn}
}

\end{document}